\documentclass[final,onefignum,onetabnum]{siamart190516}

\usepackage{amsfonts}
\usepackage{graphicx}
\usepackage{epstopdf}
\usepackage{algorithmic}
\ifpdf
  \DeclareGraphicsExtensions{.eps,.pdf,.png,.jpg}
\else
  \DeclareGraphicsExtensions{.eps}
\fi

\newsiamremark{remark}{Remark}
\newsiamremark{hypothesis}{Hypothesis}
\crefname{hypothesis}{Hypothesis}{Hypotheses}
\newsiamthm{claim}{Claim}

\usepackage[utf8]{inputenc}
\usepackage[english]{babel}
\usepackage{amsmath}
\usepackage{amssymb}
\usepackage{graphicx}
\usepackage{lmodern}
\usepackage{tikz}
\usepackage{color}
\usepackage{bm} 
\usepackage{fourier} 
\usepackage{mathtools} 
\usepackage{enumerate}

\newcommand\eqae{\stackrel{\mathclap{\normalfont\mbox{\small a.e.}}}{\quad = \quad}}

\newcommand{\defeq}{\vcentcolon=}

\headers{A Radon Transform on the Cylinder}{A. Coyoli}

\title{A Radon Transform on the Cylinder
}

\author{Alejandro Coyoli 
\thanks{Department of Mathematics, Tufts University, Medford, MA
  (\email{Alejandro.Coyoli@tufts.edu})}} 

\usepackage{amsopn}

\ifpdf
\hypersetup{
  pdftitle={A Radon Transform on the Cylinder},
  pdfauthor={A. Coyoli}
}
\fi

\begin{document}

\maketitle

\begin{abstract}
We define a parametric Radon transform $R$ that assigns to a Sobolev function on the cylinder $\mathbb{S}\times \mathbb{R}$ in $\mathbb{R}^3$ its mean values along sets $E_\zeta$ formed by the intersections of planes through the origin and the cylinder. We show that $R$ is a continuous operator, prove an inversion formula, provide a support theorem, as well as a characterization of its null space. We conclude by presenting a formula for the dual transform $R^*$. We show that $R$ and its dual $R^*$ are related to the right-sided and left-sided Chebyshev fractional integrals. Using this relationship, we characterize the null space of $R$ and $R^*$ and provide an inversion formula for $R^*$.
\end{abstract}



\section{Introduction}
The problem of reconstructing a function from its averages along submanifolds goes back to the $19^{th}$ century \cite{RubNotes2006}. The pioneering works of Paul Funk \cite{FunkLinien} and Johann Radon \cite{RadonLangs} contain fundamental ideas that not only facilitated further developments in numerous fields of mathematics and physics, but lie at the heart of modern tomography, applications to homeland security and non-destructive testing in industry, to mention just a few applications \cite{KuchMedicalIm}. 

A mapping that assigns to each sufficiently good function $f$ on a given manifold $X$ a collection of integrals of $f$ over submanifolds of $X$ is commonly called the Radon transform of $f$ \cite{GelfIntegralGeo,HelgGeneralizationsRadon,
RubNotes2006}. However, the term Radon transform has a wide meaning and is used in diverse transforms of the kind \cite{CormSpheresThroughOrigin,
QuellGeneralizationFunk,QuellFunkRadon, QuintoInjectivity}.

In this article we define a Radon-like transform $R$ that assigns to a Sobolev function on the cylinder $\mathbb{S}\times \mathbb{R}$ in $\mathbb{R}^3$ its mean values along sets $E_\zeta$ formed by the intersections of planes through the origin and the cylinder. Unlike the classical Radon transform, the parametric transform $R$ averages functions using the arclength measure $d\sigma$ on $\mathbb{S}$ and not $dS$, the arclength measure on $E_\zeta$. This simplification gives rise to a continuous and invertible transform on the cylinder with strong connections to fractional calculus. By using $d\sigma$ instead of the arclength measure $dS$ on $E_{\zeta}$ we also avoid issues that arise with the averaging of a function over sets of infinite arclength. To the best of our knowledge, this problem has not been considered in the existing literature. 

This article is organized as follows. In \Cref{Sec:Transform}, we introduce the parametric transform $R$. In \Cref{sec:suptheory}, we present some auxiliary concepts and technical results needed for later sections. \Cref{sec:properties} and \Cref{sec:Dualsec} are the core of this paper and contain our main results. 

\section{The Parametric Transform \textbf{R}}\label{Sec:Transform}

\begin{figure}\label{Geoproblem}
\begin{center}
\begin{tabular}{ccc}
\includegraphics[scale=0.2]{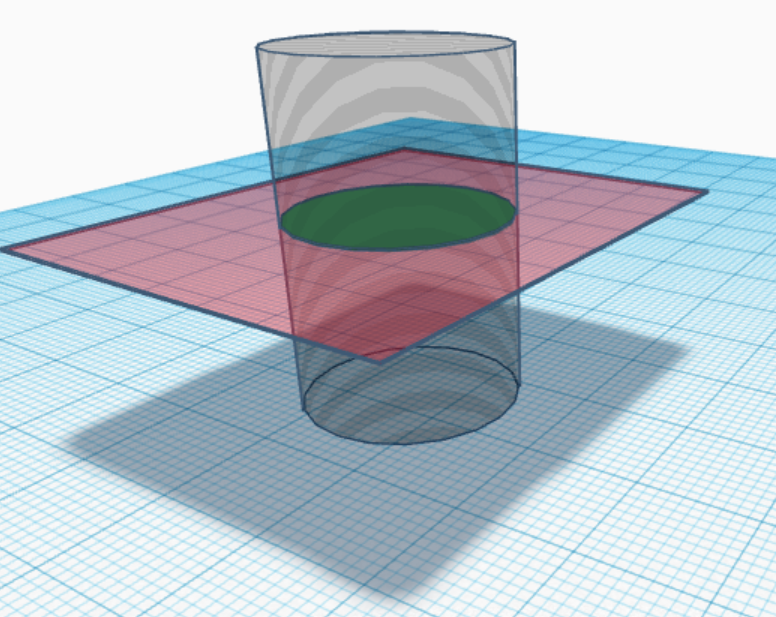} & \includegraphics[scale=0.31]{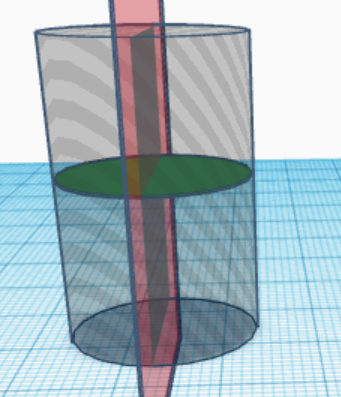} & \includegraphics[scale=0.36]{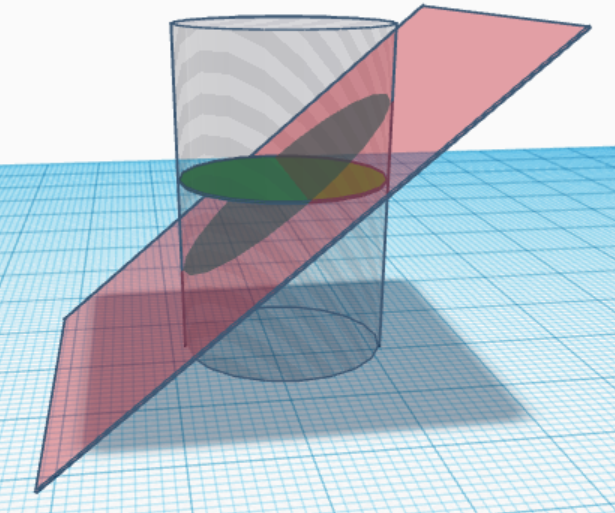} \\ 

\end{tabular} 
\end{center}
\caption{Geometry of sets $E_{\zeta(\theta,\rho)}$ for $\zeta(\theta,\rho) \in \mathbb{S}^2$ a pole, in the equator and elsewhere on $\mathbb{S}^2$.}
\end{figure}
We parametrize points on the cylinder $\mathbb{S}\times \mathbb{R}$ using local coordinates $s \in [0,2\pi) $ and $t\in \mathbb{R}$ by $(e^{is},t)$. We consider functions that are square integrable on the cylinder and whose weak partial derivatives $D^\alpha f$ exist and are square integrable for all multiindices $|\alpha|\leq 2$. We denote this function space $H^2(\mathbb{S}\times\mathbb{R})$ and note that it consists of functions that are continuous almost everywhere.

Consider the subset of $H^2(\mathbb{S}\times\mathbb{R})$
\begin{align*}
H^2_{pc}(\mathbb{S}\times\mathbb{R})= \left\lbrace \left.  f\in H^2(\mathbb{S}\times\mathbb{R}) \right| \lim_{t\to \infty} f(e^{is},t) \eqae C(s)\text{ with } C \text{ integrable, }2\pi-\text{periodic} \right\rbrace.
\end{align*}
These functions behave nicely at infinity and allow their mean values along sets $E_{\zeta}$, formed by the intersection of planes through the origin and the cylinder, to converge.

The Radon-like operator $R$ assigns to functions in $H^2_{pc}(\mathbb{S}\times \mathbb{R})$ their mean values along sets $E_\zeta$ formed by the intersections of planes through the origin and the cylinder. The intersections of the cylinder $\mathbb{S} \times \mathbb{R}$ and planes through the origin define one of three different objects on $\mathbb{S}\times\mathbb{R}$, depending on the normal to the plane $\bm{\zeta}\in \mathbb{S}^2$ (\Cref{Geoproblem}). The set of points resulting from the intersection of $\mathbb{S}\times \mathbb{R}$ and planes through the origin with normal $\bm{\zeta}$ given by a pole on $\mathbb{S}^2$ describes a circle. When $\bm{\zeta}$ lies on the equator of the sphere, the set describes two parallel lines on $\mathbb{S}\times \mathbb{R}$ perpendicular to the plane $z=0$. Finally, when the normal to the plane is neither a pole or in the equator of the sphere, the set $E_\zeta$ describes an ellipse.

We utilize the arclength measure $d\sigma$ on $\mathbb{S}$ to calculate the mean values of a function $f$ along the sets $E_{\zeta}$. Symbolically, we define the operator
\begin{align}
R:H^2_{pc}\rightarrow L^2(\mathbb{S}^2)\\
Rf(\zeta)\defeq \frac{1}{2\pi}\int_{E_{\zeta}} f d\sigma.\label{operator}
\end{align}

Granting the identification $\zeta = \zeta(\theta,\rho)=(\cos \theta \sin \rho, \sin \theta \sin \rho, \cos \rho)$ between $\mathbb{S}^2_0$, the $2-$dimensional sphere excluding the equator, and the set 
\begin{align*}
(\theta,\rho) \in \Xi =  [0,2\pi)\times [0,\pi]\setminus \{\pi/2\},
\end{align*}
we may parametrize the sets $E_{\zeta(\theta,\rho)}$ by
\begin{align*}
E_{\zeta(\theta,\rho)} &=\{ x \in \mathbb{S} \times \mathbb{R} \enspace | \enspace \langle \, \zeta(\theta,\rho), x \, \rangle = 0 \,\}\\
&=\{ (e^{is},-\tan \rho \cos(\theta -s)) \enspace | \enspace s\in [0,2\pi) \,\}.
\end{align*} 
This allows us to rewrite \cref{operator} as
\begin{align}\label{Rparametricnopihalf}
Rf(\theta,\rho) = \frac{1}{2\pi} \int_0^{2\pi} f(e^{is},-\tan \rho \cos(\theta -s))ds
\end{align}
for all $(\theta,\rho) \in \Xi$.

\begin{proposition}\label{prop2.1}
Let $f\in H^2_{pc}(\mathbb{S}\times\mathbb{R})$, then $Rf$ is well defined on $\mathbb{S}^2$.
\end{proposition}

\begin{proof}
Let $f\in H^2_{pc}(\mathbb{S}\times\mathbb{R})$. By \Cref{SobolevM} there exits a bounded, continuous function that is equal to $f$ almost everywhere on $\mathbb{S}\times\mathbb{R}$. Without loss of generality, let $f$ denote this continuous representative. Since $R$ annihilates  odd functions, assume that $f$ is even. Then,
\begin{align*}
Rf(\theta,\rho)&= \frac{1}{2\pi}\int_0^{2\pi} f(e^{is},-\tan \rho \cos(\theta-s))\,ds\\ 
&= \frac{1}{2\pi}\int_{\theta+\frac{\pi}{2}}^{\theta+\frac{5\pi}{2}} f(e^{is},-\tan \rho \cos(\theta-s))\,ds,
\end{align*}
as the integrand is a $2\pi-$periodic function of $s$. Using the fact that $f$ is an even function on $\mathbb{S}\times \mathbb{R}$ and making a change of variables $r =s-\pi$ we obtain
\begin{align*}
Rf(\theta,\rho) &=\frac{1}{\pi}\int_{\theta+\frac{\pi}{2}}^{\theta+\frac{3\pi}{2}} f(e^{is},-\tan \rho \cos(\theta-s))\,ds.
\end{align*}

By hypothesis, there exists  a $2\pi-$periodic function $C\in L^1([0,2\pi))$ such that
\begin{align*}
\lim_{t\to \infty} f(e^{is},t) \eqae C(s).
\end{align*}
Since $f$ is bounded on $\mathbb{S}\times \mathbb{R}$, by the Lebesgue Dominated Convergence Theorem,
\begin{align*}
Rf(\theta,\pi/2)&= \frac{1}{\pi} \lim_{\rho\to \pi/2} \, \int_{\theta+\frac{\pi}{2}}^{\theta+\frac{3\pi}{2}} f(e^{is},-\tan \rho \cos(\theta-s))\,ds\\
&= \frac{1}{\pi} \int_{\theta+\frac{\pi}{2}}^{\theta+\frac{3\pi}{2}} \lim_{\rho\to \pi/2}f(e^{is},-\tan\rho \cos (\theta -s))\,ds\\
&=  \frac{1}{\pi} \int_{\theta+\frac{\pi}{2}}^{\theta+\frac{3\pi}{2}}C(s)\,ds\\
&< \infty.
\end{align*}
\end{proof}
Therefore, we define the transform $R$ on the entire sphere $\mathbb{S}^2$, or equivalently in $\Xi \cup \{\pm \pi/2\}$, as 
\begin{align}\label{Rparametric}
Rf(\theta,\rho) \defeq \begin{cases} \frac{1}{2\pi}\int_0^{2\pi} f(e^{is},-\tan \rho \cos (\theta - s))ds &\quad \text{ if } (\theta,\rho) \in \Xi \\
&\\
\lim_{\varphi \to \pi/2} Rf(\theta,\varphi) &\quad \text{ if } (\theta,\rho)\in [0,2\pi)\times \{\pm \pi/2\}.\end{cases}
\end{align}
 
The following result is a corollary of \Cref{prop2.1} for a class of well behaved functions in $H^2_{pc}(\mathbb{S}\times\mathbb{R})$: functions that vanish at infinity. 
\begin{corollary}
Let $f \in H^2(\mathbb{S}\times\mathbb{R})$ be a function that vanishes at infinity. Then $Rf$ is well defined on $\mathbb{S}^2$ and $Rf(\theta,\pi/2)=0$ for all $\theta \in [0,2\pi)$.
\end{corollary}

\section{Notation, Definitions \& Technical Results}\label{sec:suptheory}
Let $\mathbb{S}\times \mathbb{R}$ be the $2-$dimensional cylinder in $\mathbb{R}^3$, $\mathbb{S}^2$ the $2-$dimensional sphere,  $\mathbb{S}^2_+ \subset \mathbb{S}^2$ the upper hemisphere of the sphere and $\mathbb{S}^2_0 \subset \mathbb{S}^2$ the sphere excluding the equator. 

Let $1 \leq p <\infty $, then $L^p(\mathbb{S}\times\mathbb{R})$ denotes the usual Lebesgue function space of complex-valued functions on the cylinder. Let $k\in\{0,1,2,...\}$, we use $W^{k,p}(\mathbb{S}\times \mathbb{R})$ to denote the set of all Sobolev functions $f\in L^p(\mathbb{S}\times \mathbb{R})$ whose weak partial derivatives $D^\alpha f$ exist and belong to $L^p(\mathbb{S}\times \mathbb{R})$ for all $|\alpha| \leq k$. We use $H^k$ to distinguish the Hilbert spaces $H^k(\mathbb{S}\times \mathbb{R})=W^{k,2}(\mathbb{S}\times \mathbb{R})$ and $H^k_{pc}(\mathbb{S}\times \mathbb{R})\subset H^k(\mathbb{S}\times \mathbb{R}) $ the subset of Sobolev functions on the cylinder that converge pointwise almost everywhere to some periodic function $C(s)\in L^1([0,2\pi))$ at infinity. 

\Cref{SobolevR} and \Cref{SobolevM} are key tools for the results in this paper, as they guarantee the existence of a continuous and bounded function that is equal almost everywhere to $f \in H^2(\mathbb{S}\times \mathbb{R})$. We may use \Cref{SobolevM} for functions defined on the cylinder, since $\mathbb{S} \times \mathbb{R}$ is a complete Riemannian manifold of bounded curvature with injectivity radius $\delta =\pi$. 

\begin{theorem}[\cite{AubNonlinearProblems}, Sobolev Embedding]\label{SobolevR}
\begin{itemize}
\item[\textit{Part I}] Let $k$ and $\ell$ be two integers $(k>\ell \geq 0)$, $p$ and $q$ two real numbers $(1 \leq q <p)$ satisfying $1/p =1/q -(k-\ell)/n$. Then, for $\mathbb{R}^n$, $W^{k,q} \subset W^{\ell,p}$ and the identity operator is continuous.
\item[\textit{Part II}] If $(k-r)/n > 1/q$, then $W^{k,q} \subset C^r_B$ and the identity operator is continuous. Here $r \geq 0$ is an integer and $C^r_B$ is the space of $C^r$ functions which are bounded as well as their derivatives of order less than or equal to $r$.
\end{itemize}
\end{theorem}
\begin{theorem}[\cite{AubNonlinearProblems}, Theorem 2.21]\label{SobolevM}
\Cref{SobolevR} holds for $M_n$ a complete manifold with bounded curvature and injectivity radius $\delta >0$. Moreover, for any $\epsilon >0$, there exists a constant $A_q(\epsilon)$ such that every $\varphi \in W^{1,q}(M_n)$ satisfies:
\begin{align*}
\|\varphi\|_p \leq [K(n,q)+\epsilon]\|\nabla \varphi \|_q + A_q(\epsilon)\|\varphi\|_q
\end{align*}
with $1/p = 1/q -1/n>0$, where $K(n,q)$ is the smallest constant having this property.
 \end{theorem}

We use the following result to obtain the inversion formula in \Cref{sec:properties}.
\begin{lemma}[\cite{CormRadiological}, Equation $16$]\label{Cormack}
Let $\ell\in \mathbb{N}\cup\{0\}$ and $r,z \in \mathbb{R}$ such that $0<z<r$, then 
\begin{align*}
rz\int_{z}^{r}\frac{T_\ell\left( \frac{p}{z}\right)T_\ell\left( \frac{p}{r}\right)}{\sqrt{r^2-p^2}\sqrt{p^2-z^2}p}dp =\frac{\pi}{2}
\end{align*}
\end{lemma}
where $T_\ell$ denotes the $\ell^{th}$ Chebyshev polynomial of the first kind
\begin{align*}
T_\ell(x)= \begin{cases} \cos(\ell \arccos x) &\qquad \text{if } |x|\leq 1, \\
\cosh(\ell \text{arcosh} x) &\qquad \text{if } x \geq  1, \\
(-1)^\ell \cosh(\ell \text{arcosh}( -x)) &\qquad \text{if } x \leq  -1.\end{cases}
\end{align*}

Another set of major auxiliary results comes from fractional integrals, within the field of fractional calculus. In this paper we employ two fractional integral operators as defined by \cite{RubFractional}, the right-sided and left-sided Chebyshev fractional integrals. The formula for the right-sided Chebyshev fractional integral is given by
\begin{align}\label{right-sided}
\left( \Upsilon_{-}^mf \right)(t) = \frac{2}{\sqrt{\pi}} \int_{t}^\infty \frac{T_m \left( \frac{t}{r}\right) f(r)r}{\sqrt{r^2-t^2}}\,dr
\end{align}
and the left-sided Chebyshev fractional integral formula is given by
\begin{align}\label{left-sided}
\left( \Upsilon_+^m f \right)(r) = \frac{2}{\sqrt{\pi}} \int_0^r \frac{T_m\left( \frac{t}{r}\right) f(t)}{\sqrt{r^2 - t^2}}\,dt.
\end{align}

We are particularly interested in knowing the conditions under which these operators converge. 

\begin{proposition}[\cite{RubFractional}, Proposition 2.46]\label{Rub246}
Let $a>0$, the integral $\left( \Upsilon_{-}^mf \right)(t)$ is finite for almost all $t>a$ under the following condition:
\begin{align}\label{condition1}
\int_a^\infty |f(t)|t^{-\eta}\,dt < \infty \quad \text{ where } \quad \eta=\begin{cases} 0 &\quad \text{if }m\text{ is even,}\\  1 &\quad \text{if }m\text{ is odd.} \end{cases}
\end{align}
\end{proposition}

\begin{proposition}[\cite{RubFractional}, Proposition 2.57]
Let $b>0$, the integral $\left( \Upsilon_{+}^mf \right)(r)$ is absolutely convergent for almost all $r<b$ under the following condition:
\begin{align}\label{condition2}
\int_0^b t^\eta |f(t)| \,dt < \infty \quad \text{ where } \quad \eta=\begin{cases} 0 &\quad \text{if }m\text{ is even,}\\  1 &\quad \text{if }m\text{ is odd.} \end{cases}
\end{align}
\end{proposition}

Another property we are interested in is injectivity of operators \cref{right-sided} and \cref{left-sided}. The following results show that the operators are not injective for for $|m| \geq 2$. In fact, for $|m| \geq 2$, \cite{RubFractional} characterizes the null space for functions such that \cref{lchiminus} and \cref{lchiplus} hold.

\begin{lemma}[\cite{RubFractional}, Lemma 2.49]\label{upsmenosinj}
If $m=0,1$, then $\Upsilon_{-}^m $ is injective on $\mathbb{R^+}$ in the class of functions satisfying \cref{condition1} for all $a>0$. If $m \geq 2$, then $\Upsilon_{-}^m $ is non-injective in this class of functions.
\end{lemma}

\begin{lemma}[\cite{RubFractional}, Lemma 2.58]\label{upsmasinj}
If $m=0,1$, then $\Upsilon_{+}^m $ is injective on $\mathbb{R^+}$ in the class of functions satisfying \cref{condition2} for all $b>0$. If $m \geq 2$, then $\Upsilon_{+}^m $ is non-injective in this class of functions.
\end{lemma}

Finally, if we let
\begin{align*}
\chi^-_m(t)&=\begin{cases} 1 &\quad \text{if }m\text{ is even},\\ t^{-1}(1+ |\log t|) &\quad \text{if }m\text{ is odd}\end{cases}\\
\chi^+_m(t)&=\begin{cases} 1 &\quad \text{if }m\text{ is even},\\ t(1+ |\log t|) &\quad \text{if }m\text{ is odd}\end{cases}
\end{align*}
and given $ a \geq 0$, we denote by $L_{\chi^-_m}^1(a,\infty)$ the set of all functions $f$ on $(a,\infty)$ such that 
\begin{align}\label{lchiminus}
\int_{a_1}^\infty |f(t)| \chi_m^{-}(t)\,dt <\infty \qquad \text{for all }  a_1 > a,
\end{align} 
and for $0<b\leq \infty$, $L_{\chi^+_m}^1(0,b)$ the set of all functions $f$ on $(0,b)$ such that 
\begin{align}\label{lchiplus}
\int_0^{b_1} |f(t)| \chi_m^{+}(t)\,dt <\infty \qquad \text{for all } b_1 < b,
\end{align}
then the following two lemmas characterize the null spaces of both operators for functions in $L^1_{\chi^-_m}(a,\infty)$ and $L^1_{\chi^+_m}(0,b)$.
\begin{lemma}[\cite{RubFractional}, Lemma 2.62]\label{upsmenos}
Let $m\in \mathbb{N}$ such that $m \geq 2$, $M = \left\lfloor \frac{m}{2} \right\rfloor$ and $a \geq 0$. If $f \in L_{\chi^-_m}^1(a,\infty)$ and $(\Upsilon_{-}^m f)(r)=0 $ for almost all $r>a$, then for some coefficients $c_k$
\begin{align*}
f(t)= \sum_{k=0}^{M-1} c_k t^{2k-m} \qquad \text{a.e. on }(a,\infty).
\end{align*}
\end{lemma}

\begin{lemma}[\cite{RubFractional}, Lemma 2.61]\label{upsmas}
Let $m\in \mathbb{N}$ such that $m \geq 2$, $M = \left\lfloor \frac{m}{2} \right\rfloor$ and $b \geq 0$. If $f \in L_{\chi^+_m}^1(0,b)$ and $(\Upsilon_{+}^m f)(r)=0 $ for almost all $0<r<b$, then for some coefficients $c_j$
\begin{align*}
f(t) =\sum_{j=1}^M c_j t^{m-2j} \qquad \text{a.e. on }(0,b).
\end{align*}
\end{lemma}

The following two results are used to invert the dual transform of $R$ in \Cref{sec:Dualsec}.

\begin{theorem}[\cite{RubFractional}, Theorem 2.44]\label{Rub244}
Let $f$ satisfy
\begin{align*}
\int_a^\infty |f(r)|dr <\infty \qquad \text{ for every }a >  0.
\end{align*}
Then 
$$f(t)= (\textsc{D}^{\frac{1}{2}}_{-,2} I^{\frac{1}{2}}_{-,2}f)(t)$$
where 
\begin{align}
\left( I^{\frac{1}{2}}_{-,2} f \right)(t)&=\frac{2}{\sqrt{\pi}} \int_t^\infty \frac{f(r)r}{\sqrt{r^2 -t^2}}dr \label{Ithing}\\
\textsc{D}^{\frac{1}{2}}_{-,2} \varphi &= -\frac{1}{2} \frac{d}{dt} t I^{\frac{1}{2}}_{-,2} t^{-2} \varphi,
\end{align}
and $t>0$.
\end{theorem}

\begin{theorem}[\cite{RubFractional}, Corollary 2.52]\label{Rub252}
Let $m\in \mathbb{N}$ such that $n \geq 2$. Let $f$ be a function such that
\begin{align*}
\int_a^\infty |f(t)|t^{m-1}\,dt \qquad \text{ for all }a > 0.
\end{align*}
Then $f(t)$ can be uniquely reconstructed for almost all $t>0$ from the Chebyshev fractional integral $\Upsilon_-^m f= g$ by the formula
\begin{align*}
f(t)= -\frac{1}{2}\frac{d}{dt} \left( \overset{*}{\Upsilon}^m_- t^{-2}g\right)(t),
\end{align*}
where
\begin{align*}
\left( \overset{*}{\Upsilon}^m_- g \right)(t)=\frac{2t}{\sqrt{\pi}}\int_t^\infty\frac{g(r)T_m\left(\frac{r}{t} \right)}{r \, \sqrt{r^2-t^2}} dr.
\end{align*}
\end{theorem}
For the cases $m=0,1$, we first notice that
\begin{align*}
\Upsilon^0_- f &= I^{\frac{1}{2}}_{-,2} f\\
\Upsilon^1_- f &= t I^{\frac{1}{2}}_{-,2} t^{-1}f
\end{align*}
where $I^{1/2}_{-,2}$ is defined by \cref{Ithing} and $t>0$. We may then find a unique solution to the equation $\Upsilon_-^{m} f = g$ for $m=0,1$ using \Cref{Rub244}. 

\section{Properties of R}\label{sec:properties}
\subsection*{Continuity}

\begin{theorem}\label{coninuityofR}
$R:H^2_{pc}(\mathbb{S} \times \mathbb{R}) \rightarrow L^2(\mathbb{S}^2)$ is a continuous transform.
\end{theorem}

\begin{proof}
Let $f \in H^2_{pc}(\mathbb{S} \times \mathbb{R}) $ and consider the norms 
\begin{align*}
\|f\|_2^2 &=\int_{\mathbb{S}\times \mathbb{R}} \left| f(x) \right|^2 dS(x)\\
&=\int_0^{2\pi} \int_{-\infty}^\infty \left| f(e^{is},t) \right|^2 dt \, ds\\
\end{align*}
and
\begin{align*}
\| Rf\|_2^2 &= \int_{\mathbb{S}^2} \left| Rf(\zeta)\right|^2 dS(\zeta)\\
&= 2 \int_0^{2\pi} \int_0^{\frac{\pi}{2}} \left| Rf(\theta,\rho)\right|^2 \sin \rho \,d\rho \,d\theta.
\end{align*}
Now,
\begin{align*}
\int_0^{2\pi} \int_0^{\frac{\pi}{2}} \left| Rf(\theta,\rho)\right|^2 \sin \rho\, d\rho \,d\theta
&\leq \int_0^{2\pi} \int_0^{\frac{\pi}{2}} \left( R|f|(\theta,\rho)\right)^2 \sin \rho \,d\rho \,d\theta.
\end{align*}
Using Jensen's inequality, we obtain that
\begin{align*}
 \left( R|f|(\theta,\rho)\right)^2 \leq \frac{1}{2\pi} \int_0^{2\pi} \left|f(e^{is},-\tan \rho \cos (\theta -s)) \right|^2 ds.
\end{align*} 
Therefore,
\begin{align*}
\int_0^{2\pi} \int_0^{\frac{\pi}{2}} \left| Rf(\theta,\rho)\right|^2 \sin \rho \, d\rho \,d\theta
&\leq \frac{1}{2\pi} \int_0^{2\pi} \int_0^{\frac{\pi}{2}} \int_0^{2\pi} \left| f(e^{is},-\tan \rho \cos (\theta -s))\right|^2 \sin \rho \, ds \,d\rho \,d\theta.
\end{align*}
Switching the order of integration, letting $\gamma = \theta -s + \pi$ with respect to $\theta$ and noticing that the integrand is $2\pi-$periodic with respect to the $\gamma$ variable, we obtain
\begin{align*}
\int_0^{2\pi} \int_0^{\frac{\pi}{2}} \left| Rf(\theta,\rho)\right|^2 \sin \rho \, d\rho \,d\theta
&\leq \frac{1}{2\pi} \int_0^{2\pi} \int_0^{\frac{\pi}{2}} \int_{0}^{2\pi} \left| f(e^{is},\tan \rho \cos \gamma)\right|^2 \sin \rho \, d\gamma \,d\rho \,ds\\
&= \frac{1}{2\pi} \int_0^{2\pi} \int_0^{\frac{\pi}{2}} \int_{0}^{\pi} \left| f(e^{is},\tan \rho \cos \gamma)\right|^2 \sin \rho \, d\gamma \,d\rho \,ds\\ &\quad + \frac{1}{2\pi} \int_0^{2\pi} \int_0^{\frac{\pi}{2}} \int_{\pi}^{2\pi} \left| f(e^{is},\tan \rho \cos \gamma)\right|^2 \sin \rho \, d\gamma \,d\rho \,ds\\
&= \frac{1}{2\pi} \int_0^{2\pi} \int_0^{\frac{\pi}{2}} \int_{0}^{\pi} \left| f(e^{is},\tan \rho \cos \gamma)\right|^2 \sin \rho \, d\gamma \,d\rho \,ds\\ &\quad +\frac{1}{2\pi} \int_0^{2\pi} \int_0^{\frac{\pi}{2}} \int_{0}^{\pi} \left| f(e^{is},-\tan \rho \cos \theta)\right|^2 \sin \rho \, d\theta \,d\rho \,ds.
\end{align*}
Noting that $f$ is even, changing the order of integration and letting $u = s+\pi$ with respect to the $s$ variable, using the fact that the integrand is $2\pi-$periodic with respect to the new variable $u$ and switching the order of integration back again we obtain that
\begin{align*}
&\int_0^{2\pi} \int_0^{\frac{\pi}{2}} \left| Rf(\theta,\rho)\right|^2 \sin \rho \, d\rho \,d\theta\\
&\qquad \leq \frac{1}{\pi} \int_0^{2\pi} \int_0^{\frac{\pi}{2}} \int_{0}^{\pi} \left| f(e^{is},\tan \rho \cos \gamma)\right|^2 \sin \rho \, d\gamma \,d\rho \,ds\\
&\qquad = \frac{1}{\pi}  \int_0^{2\pi} \int_0^{\frac{\pi}{2}} \int_{0}^{\pi} \left| f(e^{is},\tan \rho \cos \gamma)\right|^2 \sin \rho \, (\sin^2\gamma + \cos^2 \gamma)\, d\gamma \,d\rho \,ds\\
&\qquad = \frac{1}{\pi}  \int_0^{2\pi} \int_0^{\frac{\pi}{2}} \int_{0}^{\pi} \left| f(e^{is},\tan \rho \cos \gamma)\right|^2 \sin \rho \,\sin^2\gamma \, d\gamma \,d\rho \,ds\\ 
&\qquad \quad + \frac{1}{\pi}  \int_0^{2\pi} \int_0^{\frac{\pi}{2}} \int_{0}^{\pi} \left| f(e^{is},\tan \rho \cos \gamma)\right|^2 \sin \rho \, \cos^2 \gamma \, d\gamma \,d\rho \,ds\\
&\qquad \leq \frac{1}{\pi}  \int_0^{2\pi} \int_0^{\frac{\pi}{2}} \int_{0}^{\pi} \left| f(e^{is},\tan \rho \cos \gamma)\right|^2 \sin \rho \,\sin\gamma \, d\gamma \,d\rho \,ds\\ 
&\qquad \quad+ \frac{1}{\pi}  \int_0^{2\pi} \int_0^{\frac{\pi}{2}} \int_{0}^{\pi} \left| f(e^{is},\tan \rho \cos \gamma)\right|^2 \sin \rho \, |\cos \gamma| \, d\gamma \,d\rho \,ds.
\end{align*}

On the one hand,
\begin{align*}
&\frac{1}{\pi} \int_0^{2\pi} \int_0^{\frac{\pi}{2}} \int_{0}^{\pi} \left| f(e^{is},\tan \rho \cos \gamma)\right|^2 \sin \rho \,\sin\gamma \, d\gamma \,d\rho \,ds\\
 &\qquad = \frac{1}{\pi}  \int_0^{2\pi} \int_0^{\frac{\pi}{2}} \int_{0}^{\pi} \left| f(e^{is},\tan \rho \cos \gamma)\right|^2 \cos \rho \tan \rho \,\sin\gamma \, d\gamma \,d\rho \,ds\\
&\qquad \leq  \frac{1}{\pi}  \int_0^{2\pi} \int_0^{\frac{\pi}{2}} \int_{0}^{\pi} \left| f(e^{is},\tan \rho \cos \gamma)\right|^2 \tan \rho \,\sin\gamma \, d\gamma \,d\rho \,ds\\
&\qquad = \frac{1}{\pi}  \int_0^{2\pi} \int_0^{\frac{\pi}{2}} \int_{-\tan\rho}^{\tan\rho} \left| f(e^{is},t)\right|^2  \, dt \,d\rho \,ds,
\end{align*}
by letting $t=\tan \rho \cos \gamma$ with respect to the $\gamma$ variable. Therefore,
\begin{align*}
&\frac{1}{\pi}  \int_0^{2\pi} \int_0^{\frac{\pi}{2}} \int_{0}^{\pi} \left| f(e^{is},\tan \rho \cos \gamma)\right|^2 \sin \rho \,\sin\gamma \, d\gamma \,d\rho \,ds\\
&\qquad \leq \frac{1}{\pi}  \int_0^{2\pi} \int_0^{\frac{\pi}{2}} \int_{-\infty}^{\infty} \left| f(e^{is},t)\right|^2  \, dt \,d\rho \,ds\\
&\qquad =  \frac{1}{2}  \int_0^{2\pi} \int_{-\infty}^{\infty} \left| f(e^{is},t)\right|^2  \, dt \,ds.
\end{align*}

Consequently,
\begin{align*}
\frac{1}{\pi}  \int_0^{2\pi} \int_0^{\frac{\pi}{2}} \int_{0}^{\pi} \left| f(e^{is},\tan \rho \cos \gamma)\right|^2 \sin \rho \,\sin\gamma \, d\gamma \,d\rho \,ds &\leq \frac{1}{2}  \int_0^{2\pi} \int_{-\infty}^{\infty} \left| f(e^{is},t)\right|^2  \, dt \,ds.
\end{align*}

On the other hand, 
\begin{align*}
&\frac{1}{\pi}  \int_0^{2\pi} \int_0^{\frac{\pi}{2}} \int_{0}^{\pi} \left| f(e^{is},\tan \rho \cos \gamma)\right|^2 \sin \rho \, |\cos \gamma| \, d\gamma \,d\rho \,ds\\
&\qquad\leq \frac{1}{\pi}  \int_0^{2\pi} \int_0^{\frac{\pi}{2}} \int_{0}^{\pi} \left| f(e^{is},\tan \rho \cos \gamma)\right|^2 \, |\cos \gamma| \, d\gamma \,d\rho \,ds\\
&\qquad= \frac{1}{\pi}  \int_0^{2\pi} \int_0^{\frac{\pi}{2}} \int_{0}^{\frac{\pi}{2}} \left| f(e^{is},\tan \rho \cos \gamma)\right|^2 \,\cos\gamma \, d\gamma \,d\rho \,ds\\
&\qquad \quad - \frac{1}{\pi}  \int_0^{2\pi} \int_0^{\frac{\pi}{2}} \int_{\frac{\pi}{2}}^{\pi} \left| f(e^{is},\tan \rho \cos \gamma)\right|^2  \,\cos\gamma \, d\gamma \,d\rho \,ds\\
&\qquad=\frac{1}{\pi}  \int_0^{2\pi} \int_{0}^{\frac{\pi}{2}} \int_0^{\frac{\pi}{2}}  \left| f(e^{is},\tan \rho \cos \gamma)\right|^2  \,\cos\gamma \,d\rho \, d\gamma  \,ds\\
&\qquad \quad - \frac{1}{\pi}  \int_0^{2\pi} \int_{\frac{\pi}{2}}^{\pi} \int_0^{\frac{\pi}{2}}  \left| f(e^{is},\tan \rho \cos \gamma)\right|^2  \,\cos\gamma \,d\rho \, d\gamma  \,ds.
\end{align*} 

Let the change of variables $t = \tan \rho \cos \gamma$ with respect to $\rho$, then
\begin{align*}
&\frac{1}{\pi}  \int_0^{2\pi} \int_0^{\frac{\pi}{2}} \int_{0}^{\pi} \left| f(e^{is},\tan \rho \cos \gamma)\right|^2 \sin \rho \, |\cos \gamma| \, d\gamma \,d\rho \,ds\\
&\qquad \leq \frac{1}{\pi}  \int_{0}^{2\pi} \int_0^{\frac{\pi}{2}} \int_0^\infty \left| f(e^{is},t)\right|^2 \frac{\cos^2 \gamma}{\cos^2 \gamma +t^2} \, dt\, d\gamma \,ds\\
&\qquad \quad + \frac{1}{\pi}  \int_{0}^{2\pi} \int_{\frac{\pi}{2}}^{\pi} \int_{-\infty}^0 \left| f(e^{is},t)\right|^2 \frac{\cos^2 \gamma}{\cos^2 \gamma +t^2} \, dt\, d\gamma \,ds\\
&\qquad \leq \frac{1}{\pi}  \int_{0}^{2\pi} \int_0^{\frac{\pi}{2}} \int_0^\infty \left| f(e^{is},t)\right|^2  \, dt\, d\gamma \,ds\\
&\qquad \quad + \frac{1}{\pi}  \int_{0}^{2\pi} \int_{\frac{\pi}{2}}^{\pi} \int_{-\infty}^0 \left| f(e^{is},t)\right|^2 \, dt\, d\gamma \,ds\\
&\qquad = \frac{1}{2} \int_{0}^{2\pi} \int_{-\infty}^\infty \left| f(e^{is},t)\right|^2 \, dt\, ds.
\end{align*} 
Consequently,
\begin{align*}
\frac{1}{\pi}  \int_0^{2\pi} \int_0^{\frac{\pi}{2}} \int_{0}^{\pi} \left| f(e^{is},\tan \rho \cos \gamma)\right|^2 \sin \rho \, |\cos \gamma| \, d\gamma \,d\rho \,ds
&\leq  \frac{1}{2} \int_{0}^{2\pi} \int_{-\infty}^\infty \left| f(e^{is},t)\right|^2 \, dt\, ds.
\end{align*}
Therefore, putting both bounds together, we obtain 
\begin{align*}
\int_0^{2\pi} \int_0^{\frac{\pi}{2}} \left| Rf(\theta,\rho)\right|^2 \sin \rho \, d\rho \,d\theta
&\leq \int_{0}^{2\pi} \int_{-\infty}^\infty \left| f(e^{is},t)\right|^2 \, dt\, ds,
\end{align*}
or
\begin{align*}
\left \| Rf \right \|_2 \leq \sqrt{2} \left\| f \right\|_2 
\end{align*}
and therefore, $R$ is continuous from $H^2_{pc}(\mathbb{S}\times \mathbb{R})$ to $L^2(\mathbb{S}^2)$.
\end{proof}

\subsection*{Null Space, Injectivity of R and an Inversion Formula}

The transform $R$, just like the Funk transform, vanishes for all odd functions with respect to the equator. We will show that a necessary condition for a function space to have a non-trivial null space is that functions may not be bounded at infinity. To prove this we will make use of the following result.
\begin{lemma}\label{Gnlemma}
Let $f \in H^2_{loc}(\mathbb{S}\times \mathbb{R})$ and $f(e^{is},t)=\sum_{n\in\mathbb{Z}}F_n(t)e^{ins}$ its Fourier series. Let $Rf(\theta,\rho)=\sum_{n\in\mathbb{Z}}G_n(\rho)e^{in\theta}$ be the Fourier series of $Rf$ on $\Xi$. Then
\begin{align}
G_n(\arctan x)= \frac{2(-1)^{|n|} }{\pi} \int_{0}^{x}\frac{F_n(v)T_{|n|}\left( \frac{v}{x}\right)}{\sqrt{x^2-v^2}} dv
\end{align} 
where $x=\tan \rho$ and $T_{|n|}$ is the $|n|^{th}$ Chebyshev polynomial of the first kind.
\end{lemma}

\begin{proof}
Let $f\in H^2_{loc}(\mathbb{S}\times\mathbb{R})$ and $(\theta,\rho)\in \Xi$. Since $f$ is in $H^2(E_{\zeta(\theta,\rho)})$, by \Cref{SobolevM} $f$ is equal almost everywhere to a continuous and bounded function on $E_{\zeta(\theta,\rho)}$. Without loss of generality, assume that $f$ is such representative. Since $R$ annihilates odd functions, assume that $f$ is an even function. Consequently, $Rf$ is an even function on $\Xi$ in the sense that $Rf(\theta,\rho)=Rf(\theta+\pi,\pi-\rho)$. Therefore, assume, without loss of generality, that $0 \leq \rho < \pi/2$.

We parametrize the set $E_{\zeta(\theta,\rho)}$ as in \Cref{Sec:Transform} so that
\begin{align*}
Rf(\theta,\rho)&= \frac{1}{2\pi} \int_{E_{\zeta(\theta,\rho)}} f(x)d\sigma(x)\\
&=\frac{1}{2\pi}\int_{0}^{2\pi}f(e^{is},-\tan\rho\cos(\theta - s))ds.
\end{align*}
 
Let $G_n(\rho)$ be the the $n^{th}$ Fourier coefficient of $Rf$. Then
\begin{align*}
G_n(\rho) &= \frac{1}{2\pi}\int_0^{2\pi} Rf(\theta,\rho)e^{-in\theta}d\theta\\
&=\frac{1}{4\pi^2}\int_0^{2\pi} \int_0^{2\pi} f(e^{is},-\tan \rho \cos(\theta - s))e^{-in\theta}ds\,d\theta.
\end{align*}
Letting $u=s-\theta$, using the fact that the integrand is $2\pi-$periodic with respect to $u$, switching the order of integration, since $f$ is bounded on $E_{\zeta(\theta,\rho)}$, and letting $s= \theta + u$ we obtain
\begin{align*}
G_n(\rho)&=\frac{1}{2\pi} \int_{0}^{2\pi}F_n(-\tan\rho\cos u)e^{-inu} du\\
&= \frac{1}{2\pi} \left( \int_{0}^{\pi}F_n(-\tan\rho\cos u)e^{-inu} du +\int_{\pi}^{2\pi}F_n(-\tan\rho\cos u)e^{-inu} du\right)\\
&=\frac{1}{2\pi} \left( \int_{0}^{\pi}F_n(-\tan\rho\cos u)e^{-inu} du +\int_{0}^{\pi}F_n(-\tan\rho\cos(v+\pi))e^{-in(v+\pi))} dv \right)\\
&= \frac{1}{2\pi} \left( \int_{0}^{\pi}F_n(-\tan\rho\cos u)e^{-inu} du +\int_{0}^{\pi} (-1)^{|n|}F_n(\tan\rho\cos v)e^{-inv} dv\right).\\
\end{align*} By hypothesis $f$ is an even function on $\mathbb{S}\times \mathbb{R}$, therefore $F_n(t)=(-1)^{|n|}F_n(-t)$ and
\begin{align*}
G_n(\rho)&=\frac{1}{\pi} \int_{0}^{\pi}F_n(-\tan\rho\cos u)e^{-inu} du.
\end{align*} 

The change of variables $v = -\tan\rho\cos u$ yields
\begin{align*}
G_n(\rho)&=\frac{1}{\pi}\int_{0}^{\pi}F_n(-\tan\rho\cos u)e^{-inu} du\\
&=\frac{1}{\pi}\int_{-\tan\rho}^{\tan\rho}F_n(v)e^{-in \arccos\left( \frac{-v}{\tan\rho}\right)}\frac{1}{\sqrt{\tan^2\rho -v^2}} dv\\
&=\frac{1}{\pi}\int_{-\tan\rho}^{\tan\rho}F_n(v)e^{-in\left(\pi - \arccos\left( \frac{v}{\tan\rho}\right)\right)}\frac{1}{\sqrt{\tan^2\rho -v^2}} dv\\
&=\frac{(-1)^{|n|} }{\pi}\int_{-\tan\rho}^{\tan\rho}F_n(v)e^{in\arccos\left( \frac{v}{\tan\rho}\right)}\frac{1}{\sqrt{\tan^2\rho -v^2}} dv.
\end{align*}Using Euler's formula and the fact that   $F_n(v)\sin(n \arccos(v/\tan \rho))/\sqrt{\tan^2\rho -v^2}$ is an odd function of $v$ yields
\begin{align*}
G_n(\rho)& =\frac{(-1)^{|n|} }{\pi} \int_{-\tan\rho}^{\tan\rho}F_n(v)T_{|n|}\left( \frac{v}{\tan\rho}\right)\frac{1}{\sqrt{\tan^2\rho -v^2}} dv\\
&=
\frac{2(-1)^{|n|} }{\pi} \int_{0}^{\tan\rho}F_n(v)T_{|n|}\left( \frac{v}{\tan\rho}\right)\frac{1}{\sqrt{\tan^2\rho -v^2}} dv,
\end{align*}
where $T_{|n|}$ is the $|n|^{th}$ Chebyshev polynomial. 

Finally, letting $x=\tan\rho$
\begin{align}
G_n(\arctan x)= \frac{2(-1)^{|n|} }{\pi} \int_{0}^{x}\frac{F_n(v)T_{|n|}\left( \frac{v}{x}\right)}{\sqrt{x^2-v^2}} dv.
\end{align} 
\end{proof}

We obtain the following null space characterization when we define $R$ on the set of continuous functions on the cylinder.

\begin{theorem}
Let $R:C(\mathbb{S} \times \mathbb{R})  \rightarrow C(\mathbb{S}^2_0)$. Then the null space of $R$ consists of all odd functions in $C(\mathbb{S} \times \mathbb{R})$ and those that are even and are such that their Fourier series $f(e^{is},t)=\sum_{n\in \mathbb{Z}} F_n(t)e^{ins}$ have polynomial coefficients
\begin{align*}
F_n(t)=0
\end{align*} for $|n|< 2 $,
and
\begin{align*}
F_n(t) \eqae \sum_{m=1}^{\lfloor \frac{|n|}{2} \rfloor} a_m^n t^{|n|-2m}
\end{align*}
for $|n|\geq2$, where $a_m^n \in \mathbb{C}$.
\end{theorem}

\begin{proof}

Let $(\theta,\rho)\in \Xi$, $x=\tan\rho$ and $f \in C(\mathbb{S} \times \mathbb{R})$. Consider the Fourier series $Rf(\theta,\rho) = \sum_{n\in \mathbb{Z}}G(\rho)e^{in\theta}$. Since $R$ annihilates odd functions, assume, without loss of generality, that $f$ is an even function. Hence, $Rf$ is an even function on $\Xi$ in the sense that $Rf(\theta,\rho)=Rf(\theta+\pi,\pi-\rho)$. Therefore, assume, without loss of generality, that $0 \leq \rho < \pi/2$ so that $x>0$. 

By \Cref{Gnlemma}
\begin{align*}
G_n(\arctan x)= \frac{2(-1)^{|n|} }{\pi} \int_{0}^{x}\frac{F_n(u)T_{|n|}\left( \frac{u}{x}\right)}{\sqrt{x^2-u^2}} du.
\end{align*}

The right hand-side of the last equation may be expressed as a left-sided Chebyshev fractional integral as in \Cref{upsmas} yielding
\begin{align*}
G_n(\arctan x)=\frac{(-1)^{|n|}}{\sqrt{\pi}}\left( \Upsilon_+^{|n|} F_n\right)(x).
\end{align*} 
Therefore,
\begin{align*}
Rf(\theta,\rho)=0 &\iff \forall n \in \mathbb{Z} \quad G_n(\arctan x) =0.
\end{align*}

By \Cref{upsmasinj}, $\Upsilon^{|n|}_+$ is injective for $|n|<2$. So for $|n|<2$, $G_n(\arctan x)= 0$ if and only if $F_n\equiv 0$. For $|n|\geq 2$, by \Cref{upsmas}
\begin{align*}
F_n(t) \eqae \sum_{m=1}^{\lfloor \frac{|n|}{2} \rfloor} a_m^n t^{|n|-2m}
\end{align*}
for some $a_m^n \in \mathbb{C}$.
\end{proof}

\begin{corollary}
The restriction of $R$ to the space of even functions in $H^2_{pc}$ is injective.
\end{corollary}

\begin{proof}
All even functions $f \in H^2_{pc}(\mathbb{S}\times\mathbb{R})$ admit a bounded continuous representative by \Cref{SobolevM}. Therefore, $f$ has bounded Fourier coefficients $ F_n \in  L^1_{\chi^+_{|n|}}(0,\infty)$ for all $n \in \mathbb{Z}$. Since the Fourier coefficients of functions in the nontrivial null space of $R$ are polynomials, therefore unbounded, this show that $R$ restricted to all even functions in $H^2_{pc}(\mathbb{S}\times\mathbb{R})$ has a trivial null space. 
\end{proof}

\begin{theorem}\label{med}
Let $f \in H^2_{loc}(\mathbb{S}\times \mathbb{R})$ and the Fourier series of $f$ on the cylinder be given by $ f(e^{is},t)= \sum_{n\in \mathbb{Z}}F_n(t)e^{ins}$. Assume for some $\epsilon>0$, there are positive constants $C_k$, $k=0,1,...$ such that
\begin{align}
|F_n(t)| \label{Fcoefs}\leq \begin{cases}  C_0 &\quad \text{ if } n=0, \\
C_n |t|^{|n|-1} &\quad \text{ if } n \in \mathbb{Z}\setminus \{0\},  \end{cases} 
\end{align}
whenever $|t|<\epsilon$. Then $f$ can be recovered almost everywhere from the averages
\begin{align*}
Rf(\theta,\rho)=\frac{1}{2\pi}\int_{0}^{2\pi} f(e^{is},-\tan \rho \cos(\theta-s))ds
\end{align*}
over all $\,(\theta,\rho) \in \Xi$. Moreover, for almost all $(e^{is},t)\in \mathbb{S}\times \mathbb{R} $ the function $f$ is given by
\begin{align*}
f(e^{is},t) = \sum_{n\in\mathbb{Z}}\left( (-1)^{|n|}\frac{d}{dt}\int_{0}^{t}\frac{G_n(\arctan x)T_{|n|}\left( \frac{t}{x}\right)x}{\sqrt{t^2 -x^2}}dx \right)e^{ins},
\end{align*}
where $G_n$ is the $n^{th}$ Fourier coefficient of $Rf$ as in \Cref{Gnlemma}.
\end{theorem}

\begin{proof}

Let $f\in H^2_{loc}(\mathbb{S}\times\mathbb{R})$. Since $R$ annihilates odd functions, assume that $f$ is an even function. Consequently, $Rf$ is an even function on $\Xi$ in the sense that $Rf(\theta,\rho)=Rf(\theta+\pi,\pi-\rho)$. Therefore, assume, without loss of generality, that $0 \leq \rho < \pi/2$ so that $x =\tan \rho \geq 0$.

Let $Rf(\theta,\rho)=\sum_{n\in\mathbb{Z}} G_n(\rho)e^{in\theta}$. From \Cref{Gnlemma}
\begin{align}\label{toinvert}
G_n(\arctan x)= \frac{2(-1)^{|n|} }{\pi} \int_{0}^{x}\frac{F_n(u)T_{|n|}\left( \frac{u}{x}\right)}{\sqrt{x^2-u^2}} du,
\end{align} 
where $x = \tan \rho$ and $T_{|n|}$ is the $|n|^{th}$ Chebyshev polynomial of the first kind.

Multiplying both sides of equation \cref{toinvert} by $T_{|n|}\left( \frac{t}{x}\right)x/\sqrt{t^2 -x^2}$ and integrating with respect to $x$ from $0$ to some positive value $t$ yields
\begin{align*}
\int_{0}^{t}\frac{G_n(\arctan x)T_{|n|}\left( \frac{t}{x}\right)x}{\sqrt{t^2 -x^2}}dx &= \frac{2(-1)^{|n|} }{\pi}\int_{0}^{t} \int_{0}^{x}\frac{F_n(u)T_{|n|}\left( \frac{u}{x}\right)T_{|n|}\left( \frac{t}{x}\right)x}{\sqrt{t^2 -x^2}\sqrt{x^2-u^2}}dudx.
\end{align*}
Note that, by \cref{Fcoefs} the left-hand side is convergent and since $f \in H^2(\mathbb{S}\times [-t,t])$, by \Cref{SobolevM} $f$ is equal almost everywhere to a continuous and bounded function. Without loss of generality, we may assume that $f$ is such representative.

To change the order of integration, let us show that
\begin{align*}
\int_{0}^{t} \int_{0}^{x} \left| \frac{F_n(u)T_{|n|}\left( \frac{u}{x}\right)T_{|n|}\left( \frac{t}{x}\right)x}{\sqrt{t^2 -x^2}\sqrt{x^2-u^2}}\right|du\,dx < \infty.
\end{align*}
Without loss of generality, we may assume that $t<\epsilon$. Otherwise, we may split the integral into two parts: one where integration occurs over $x\in [\epsilon,t]$ and another one over $x\in (0,\epsilon)$. For the integral where $x$ is bounded away from zero, by continuity of $F_n$ and $T_{|n|}$, the numerator is bounded. The resulting integral is convergent and can be calculated using trigonometric substitution.

Since for all $n \in \mathbb{Z}$
\begin{align*}
\left| T_{|n|}(x) \right| \leq \begin{cases} 1 \qquad &\text{ for } |x| \leq 1, \\ |2x|^{|n|} \qquad &\text{ for } |x| > 1, \end{cases}
\end{align*}
we have that
\begin{align*}
\left| \frac{F_n(u)T_{|n|}\left( \frac{u}{x}\right)T_{|n|}\left( \frac{t}{x}\right)x}{\sqrt{t^2 -x^2}\sqrt{x^2-u^2}} \right|
&\leq \left| \frac{F_n(u)\left( \frac{2t}{x}\right)^{|n|}x}{\sqrt{t^2 -x^2}\sqrt{x^2-u^2}} \right|.
\end{align*}
Noticing that $0 \leq x \leq t$ we obtain the following bounds.
\begin{align*}
\left| \frac{F_n(u)T_{|n|}\left( \frac{u}{x}\right)T_{|n|}\left( \frac{t}{x}\right)x}{\sqrt{t^2 -x^2}\sqrt{x^2-u^2}} \right| &\leq \begin{cases} \left| \frac{F_n(u)(2t)}{\sqrt{t^2 -x^2}\sqrt{x^2-u^2}} \right| &\quad \text{ for } |n|<2, \\
&\quad \\
\left| \frac{F_n(u)(2t)^{|n|}}{x^{|n|-1}\sqrt{t^2 -x^2}\sqrt{x^2-u^2}} \right| &\quad \text{ for }|n| \geq 2.  \end{cases}
\end{align*}
By hypothesis, there exists a value $\epsilon > 0$ for which
\begin{align*}
|F_n(t)|\leq \begin{cases}  C_0 &\quad \text{ if } n=0, \\
C_n |t|^{|n|-1} &\quad \text{ if } n \in \mathbb{Z}\setminus \{0\}, \end{cases}
\end{align*}
for $t<\epsilon$. Hence, for all $|n|<2$  
\begin{align*}
\int_{0}^{t} \int_{0}^{x} \left| \frac{F_n(u)x}{\sqrt{t^2 -x^2}\sqrt{x^2-u^2}}\right|du\,dx & \leq (2t)     
\int_{0}^{t} \frac{1}{\sqrt{t^2 -x^2}}\int_{0}^{x} \frac{\left| F_n(u) \right|}{\sqrt{x^2-u^2}}\,du\,dx\\
&\leq (2t)C_n \int_{0}^{t} \frac{1}{\sqrt{t^2 -x^2}} \int_{0}^{x} \frac{1}{\sqrt{x^2-u^2}}\,du\,dx\\
&\leq \frac{t \pi^2 C_n  }{2}\\
&< \infty.
\end{align*}
Similarly, for $|n|\geq 2$
\begin{align*}
\int_{0}^{t} \int_{0}^{x}\left| \frac{F_n(u)T_{|n|}\left( \frac{u}{x}\right)T_{|n|}\left( \frac{t}{x}\right)x}{\sqrt{t^2 -x^2}\sqrt{x^2-u^2}}\right| du\,dx &\leq (2t)^{|n|}\int_{0}^{t} \frac{1}{x^{|n|-1}\sqrt{t^2 -x^2}}\int_{0}^{x} \frac{\left| F_n(u) \right|}{\sqrt{x^2-u^2}}\,du\,dx\\
&= (2t)^{|n|} C_n \int_{0}^{t} \frac{1}{x^{|n|-1}\sqrt{t^2 -x^2}}\int_{0}^{x} \frac{u^{|n|-1}}{\sqrt{x^2-u^2}}du\,dx.
\end{align*}
The changes of variables $v = \sqrt{x^2-u^2}$ yields
\begin{align*}
\int_{0}^{t} \int_{0}^{x}\left| \frac{F_n(u)T_{|n|}\left( \frac{u}{x}\right)T_{|n|}\left( \frac{t}{x}\right)x}{\sqrt{t^2 -x^2}\sqrt{x^2-u^2}}\right| du\,dx &\leq (2t)^{|n|} C_n \int_{0}^{t} \frac{1}{x^{|n|-1}\sqrt{t^2 -x^2}}\int_{0}^{x} (x^2-v^2)^{\frac{|n|-2}{2}}dv\,dx \\
&\leq (2t)^{|n|} C_n \int_{0}^{t} \frac{1}{x^{|n|-1}\sqrt{t^2 -x^2}}\int_{0}^{x} x^{|n|-2}dv\,dx \\
&=(2t)^{|n|} C_n \int_{0}^{t} \frac{1}{\sqrt{t^2 -x^2}}dx\\
&=\frac{ (2t)^{|n|}  \pi C_n}{2}\\
&< \infty.
\end{align*}

Therefore, we may interchange the order of integration to obtain
\begin{align*}
\int_{0}^{t}\frac{G_n(\arctan x)T_{|n|}\left( \frac{t}{x}\right)x}{\sqrt{t^2 -x^2}}dx &= \frac{2(-1)^{|n|} }{\pi}\int_{0}^{t}F_n(u)\left( \int_{u}^{t}\frac{T_{|n|}\left( \frac{u}{x}\right)T_{|n|}\left( \frac{t}{x}\right)x}{\sqrt{t^2 -x^2}\sqrt{x^2-u^2}}dx\right) du.
\end{align*}
Make the change of variables $v = u^2/x$ to obtain
\begin{align*}
\int_{0}^{t}\frac{G_n(\arctan x)T_{|n|}\left( \frac{t}{x}\right)x}{\sqrt{t^2 -x^2}}dx &= \frac{2(-1)^{|n|} }{\pi}\int_{0}^{t}F_n(u)\left( u\frac{u^2}{t}\int_{\frac{ u^2}{t}}^{u}\frac{T_{|n|}\left( \frac{v}{u}\right)T_{|n|}\left( \frac{v}{(u^2/t)}\right)}{\sqrt{u^2-v^2}\sqrt{v^2 -\left(\frac{u^2}{t}\right)^2}v}dv\right) du.
\end{align*}

By \Cref{Cormack}, the expression in parentheses is equal to $\pi/2$. Therefore, 
\begin{align*}
(-1)^{|n|}\int_{0}^{t}\frac{G_n(\arctan x)T_{|n|}\left( \frac{t}{x}\right)x}{\sqrt{t^2 -x^2}}dx = \int_{0}^{t}F_n(u)du.
\end{align*}
Using the Fundamental Theorem of Calculus, we obtain the following expression for the  $n^{th}$ Fourier coefficient of $f$:
\begin{align}\label{coefformula}
F_n(t)= (-1)^{|n|}\frac{d}{dt}\int_{0}^{t}\frac{G_n(\arctan x)T_{|n|}\left( \frac{t}{x}\right)x}{\sqrt{t^2 -x^2}}dx.
\end{align} Substituting in the Fourier series of $f$ yields
\begin{align}\label{formula}
f (e^{is},t)= \sum_{n\in \mathbb{Z}} \left(  (-1)^{|n|}\frac{d}{dt}\int_{0}^{t}\frac{G_n(\arctan x)T_{|n|}\left( \frac{t}{x}\right)x}{\sqrt{t^2 -x^2}}dx \right)e^{ins}.
\end{align} 
\end{proof}

\subsection*{Support Theorem}

Notice that \Cref{med} depends only on the local behavior of $f$ near the equator and is independent of its behavior at infinity. Similarly, the following support theorem depends only on $Rf$ vanishing for $\rho \in [0,\arctan t]$ to determine that $f(e^{is},t)=0$ for all $s \in [0,2\pi)$.

\begin{theorem}[Support Theorem]\label{supporttheo}
Let $f \in H^2_{loc}(\mathbb{S}\times \mathbb{R})$ and the Fourier series of $f$ on the cylinder be given by $ f(e^{is},t)= \sum_{n\in \mathbb{Z}}F_n(t)e^{ins}$. Assume for some $\epsilon>0$, there are positive constants $C_k$, $k=0,1,...$ such that
\begin{align*}
|F_n(t)|\leq \begin{cases}  C_0 &\quad \text{ if } n=0, \\
C_n |t|^{|n|-1} &\quad \text{ if } n \in \mathbb{Z}\setminus \{0\},  \end{cases} 
\end{align*}
whenever $|t|<\epsilon$. Then for a fix $\theta \in [0,2\pi)$ and $t \in [0,\infty)$ and all $s\in [0,2\pi)$, the values of $f(e^{is},t)$ are completely determined by all values $Rf(\theta,\rho)$ for $\rho \in [0,\arctan t]$. In particular, if $Rf(\theta,\rho)\equiv 0$ for all $\rho \in [0,\arctan t]$, then $f(e^{is},t)=0$ for all $s\in [0,2\pi)$.
\end{theorem}

\begin{proof}
By \Cref{med} 
\begin{align*}
f(e^{is},t) = \sum_{n\in \mathbb{Z}} \left(  (-1)^{|n|}\frac{d}{dt}\int_{0}^{t}\frac{G_n(\arctan x)T_{|n|}\left( \frac{t}{x}\right)x}{\sqrt{t^2 -x^2}}dx \right)e^{ins}.
\end{align*}
Now
\begin{align*}
Rf(\theta,\rho)\equiv 0 \text{ for all } \rho \in [0,\arctan t] &\iff Rf(\theta,\rho) = \sum_{n \in \mathbb{Z}} G_n(\rho)e^{in\theta} = 0  \text{ for all } \rho \in [0,\arctan t]\\
 &\iff G_n(\rho) \equiv 0 \text{ for all } n \in \mathbb{Z} \text{ and } \rho \in [0,\arctan t].
\end{align*}
Thus,
\begin{align*}
\int_{0}^{t}\frac{G_n(\arctan x)T_{|n|}\left( \frac{t}{x}\right)x}{\sqrt{t^2 -x^2}}dx = 0
\end{align*}
and $f(e^{is},t)=0$ for all $s \in [0,2\pi)$.
\end{proof}

\section{The Dual Transform}\label{sec:Dualsec}
In this section we study the dual transform $R^*$. The dual will be defined as the formal adjoint of $R$ in the sense that, for $g \in C(\mathbb{S}^2)$ and $f \in H^2_{pc}(\mathbb{S}\times \mathbb{R})$, $R^*g$ satisfies the duality relation
\begin{align*}
\int_0^{2\pi}\int_0^\pi Rf(\theta,\rho)\overline{g(\theta,\rho)}|\sin \theta| \,d\rho\, d\theta = \int_0^{2\pi}\int_{-\infty}^\infty f(e^{is},t)\overline{R^*g(e^{is},t)} \,dt \, ds. 
\end{align*}
The following theorem gives us an explicit formula for $R^*$.

\begin{theorem}\label{Dual}
Let $g \in C(\mathbb{S}^2)$ be an even function and $f \in H^2_{pc}(\mathbb{S}\times \mathbb{R})$. Then the formal adjoint of $R$ is given by the formula
\begin{align*}
R^*g(e^{is},t)&=\frac{1}{\pi} \sum_{\sigma = \pm 1} \, \int^{\infty}_{0} g\left(s + \sigma \arccos \left( \frac{-t}{\sqrt{v^2 + t^2}}\right),\arctan \sqrt{v^2 + t^2} \right)\, \frac{dv}{(1+v^2 +t^2)^{\frac{3}{2}}}
\end{align*}
for $(e^{is},t)\in \mathbb{S}\times \mathbb{R}$.

\end{theorem}

\begin{proof}
Let $f\in H^2_{pc}(\mathbb{S}\times\mathbb{R})$. Since $R$ annihilates odd functions and by \Cref{SobolevM}, 
without loss of generality, assume that $f$ is a continuous, bounded  and even function on $\mathbb{S}\times \mathbb{R}$. Let $g \in C(\mathbb{S}^2)$ be even in the sense that 
\begin{align*}
g(\theta,\rho)=g(\theta+\pi, \pi - \rho)
\end{align*}
for all $\theta \in [0,2\pi)$ and $\rho \in [0,\pi]$. Then,
\begin{align*}
 \langle Rf, g \rangle_{\mathbb{S}^2} &= \int_0^{2\pi} \int_{0}^{\pi}Rf(\theta,\rho)\overline{g(\theta,\rho)} \, \sin\rho \, d\rho \, d\theta\\
 &=\frac{1}{2\pi}\int_0^{2\pi} \int_{0}^{\pi} \int_0^{2\pi} f\left( e^{is},-\tan \rho \cos(\theta-s)\right)\overline{g(\theta,\rho)} \, \sin\rho \,ds \, d\rho \, d\theta\\
 &=\frac{1}{2\pi}\int_0^{2\pi} \int_{0}^{\pi} \int_{-s}^{2\pi-s} f\left( e^{is},-\tan \rho \cos u \right)\overline{g(u+s,\rho)} \, \sin\rho \,du \, d\rho \, ds,
\end{align*}
by switching the order of integration and letting $u=\theta-s$ with respect to the $\theta$ variable.

Since both $\cos u$ and $g(u+s,\rho)$  are $2\pi-$periodic functions of $u$, we can change the inner most limits of integration to obtain
\begin{align*}
\langle Rf, g \rangle_{\mathbb{S}^2}
 &=\frac{1}{2\pi}\int_0^{2\pi} \int_{0}^{\pi} \int_{0}^{2\pi} f\left( e^{is},-\tan \rho \cos u\right)\overline{g(u+s,\rho)} \, \sin\rho \,du \, d\rho \, ds\\
 &=\frac{1}{2\pi}\int_0^{2\pi} \int_{0}^{\pi} \int_{0}^{\pi} f\left( e^{is},-\tan \rho \cos u \right)\overline{g(u+s,\rho)} \, \sin\rho \,du \, d\rho \, ds \\
 &\quad +\frac{1}{2\pi}\int_0^{2\pi} \int_{0}^{\pi} \int_{\pi}^{2\pi} f\left( e^{is},-\tan \rho \cos u\right)\overline{g(u+s,\rho)} \, \sin\rho \,du \, d\rho \, ds.
\end{align*}
Since $f$ is even, letting $v = u -\pi$ in the last integral yields 
\begin{align*}
 \langle Rf, g \rangle_{\mathbb{S}^2} &=\frac{1}{2\pi}\int_0^{2\pi} \int_{0}^{\pi} \int_{0}^{\pi} f\left( e^{is},-\tan \rho \cos u \right)\overline{g(u+s,\rho)} \, \sin\rho \,du \, d\rho \, ds \\ 
&\quad +\frac{1}{2\pi}\int_0^{2\pi} \int_{0}^{\pi} \int_{0}^{\pi} f\left( e^{i(s+\pi)},-\tan \rho \cos v \right)\overline{g(v+(s+\pi),\rho)} \, \sin\rho \,dv \, d\rho \, ds. 
\end{align*}
By switching the order of integration, letting $\omega=s+\pi$ with respect to the $s$ variable, using the fact that the integrand is $2\pi-$periodic with respect to the new variable $\omega$, switching back the order of integration and renaming the $\omega$ variable $s$ we obtain
\begin{align*}
\langle Rf, g \rangle_{\mathbb{S}^2} &= \frac{1}{\pi} \int_0^{2\pi} \int_{0}^{\pi} \int_{0}^{\pi} f\left( e^{is},-\tan \rho \cos u \right)\overline{g(s+u,\rho)} \, \sin\rho \,du \, d\rho \, ds 
\end{align*}
Letting $t = - \tan \rho \cos u$ with respect to the $u$ variable yields
\begin{align*}
\langle Rf, g \rangle_{\mathbb{S}^2} &= \frac{1}{\pi}  \int_0^{2\pi} \int_{0}^{\pi} \int_{-\tan \rho}^{\tan \rho} f\left( e^{is}, t\right)\overline{g\left(s + \arccos \left( \frac{-t}{\tan \rho}\right),\rho\right)} \, \frac{\sin \rho}{\sqrt{\tan^2\rho -t^2}} \,dt \, d\rho \, ds\\ 
&= \frac{1}{\pi}  \int_0^{2\pi} \int_{0}^{\pi/2} \int_{-\tan \rho}^{\tan \rho} f\left( e^{is}, t\right)\overline{g\left(s + \arccos \left( \frac{-t}{\tan \rho}\right),\rho\right)} \, \frac{\sin \rho}{\sqrt{\tan^2\rho -t^2}} \,dt \, d\rho \, ds\\ &\quad - \frac{1}{\pi}  \int_0^{2\pi} \int_{\pi/2}^{\pi} \int_{\tan \rho}^{-\tan \rho} f\left( e^{is}, t\right)\overline{g\left(s + \arccos \left( \frac{-t}{\tan \rho}\right),\rho\right)} \, \frac{\sin \rho}{\sqrt{\tan^2\rho -t^2}} \,dt \, d\rho \, ds\\
\end{align*}
Splitting the innermost integrals, switching the order of integration and noting that for $t<0$ we have that $-\arctan t = \arctan |t|$ and $\arctan t = -\arctan |t|$ yields
\begin{align*}
\langle Rf, g \rangle_{\mathbb{S}^2} 
&= \frac{1}{\pi}  \int_0^{2\pi} \int_{-\infty}^{0} \int_{\arctan| t|}^{\pi/2} f\left( e^{is}, t\right)\overline{g\left(s + \arccos \left( \frac{-t}{\tan \rho}\right),\rho\right)} \, \sin \rho\,\frac{ \, d\rho \,dt \, ds}{\sqrt{\tan^2\rho -t^2}}  \\ 
&\quad + \frac{1}{\pi}  \int_0^{2\pi} \int_{0}^{\infty} \int_{\arctan | t| }^{\pi/2} f\left( e^{is}, t\right)\overline{g\left(s + \arccos \left( \frac{-t}{\tan \rho}\right),\rho\right)} \, \sin \rho \, \frac{\, d\rho \,dt \, ds}{\sqrt{\tan^2\rho -t^2}}  \\
&\quad - \frac{1}{\pi}  \int_0^{2\pi} \int_{-\infty}^{0} \int_{\pi/2}^{\pi-\arctan | t|} f\left( e^{is}, t\right)\overline{g\left(s + \arccos \left( \frac{-t}{\tan \rho}\right),\rho\right)} \, \sin \rho \,\frac{\, d\rho \,dt \, ds}{\sqrt{\tan^2\rho -t^2}}  \\ 
&\quad - \frac{1}{\pi}  \int_0^{2\pi} \int_{0}^{\infty} \int_{\pi/2}^{\pi-\arctan | t|} f\left( e^{is}, t\right)\overline{g\left(s + \arccos \left( \frac{-t}{\tan \rho}\right),\rho\right)}\, \sin \rho \, \frac{\, d\rho \,dt \, ds}{\sqrt{\tan^2\rho -t^2}}  
\end{align*} 
Therefore,
\begin{align*}
\langle Rf, g \rangle_{\mathbb{S}^2} &= \frac{1}{\pi}  \int_0^{2\pi} \int_{-\infty}^{\infty} \int_{\arctan| t|}^{\pi/2} f\left( e^{is}, t\right)\overline{g\left(s + \arccos \left( \frac{-t}{\tan \rho}\right),\rho\right)} \, \sin \rho\,\frac{\, d\rho \,dt  \, ds}{\sqrt{\tan^2\rho -t^2}} \\
 &\quad - \frac{1}{\pi}  \int_0^{2\pi} \int_{-\infty}^{\infty} \int_{\pi/2}^{\pi-\arctan | t|} f\left( e^{is}, t\right)\overline{g\left(s + \arccos \left( \frac{-t}{\tan \rho}\right),\rho\right)} \, \sin \rho\,\frac{\, d\rho \,dt  \, ds}{\sqrt{\tan^2\rho -t^2}} 
\end{align*} 
Letting $u = \tan \rho$ we obtain
\begin{align*}
&\langle Rf, g \rangle_{\mathbb{S}^2}\\
&= \frac{1}{\pi}  \int_0^{2\pi} \int_{-\infty}^{\infty} \int_{|t|}^{\infty} f \left( e^{is}, t\right)\overline{g\left(s + \arccos \left( \frac{-t}{u}\right),\arctan u \right)} \, \frac{u \, du \,dt  \, ds}{(1+u^2)^{\frac{3}{2}}\sqrt{u^2 -t^2}} \\ 
&\quad - \frac{1}{\pi}  \int_0^{2\pi} \int_{-\infty}^{\infty} \int_{-\infty}^{-|t|} f \left( e^{is}, t\right)\overline{g\left(s + \arccos \left( \frac{-t}{u}\right),\arctan u \right)} \, \frac{u \, du \,dt  \, ds}{(1+u^2)^{\frac{3}{2}}\sqrt{u^2 -t^2}}\\
&= \frac{1}{\pi}  \int_0^{2\pi} \int_{-\infty}^{\infty} \int_{|t|}^{\infty} f \left( e^{is}, t\right)\overline{g\left(s + \arccos \left( \frac{-t}{u}\right),\arctan u \right)} \, \frac{u \, du \,dt  \, ds}{(1+u^2)^{\frac{3}{2}}\sqrt{u^2 -t^2}}\\ 
&\quad - \frac{1}{\pi}  \int_0^{2\pi} \int_{-\infty}^{\infty} \int_{-\infty}^{-|t|} f \left( e^{is}, t\right)\overline{g\left(s + \pi + \arccos \left( \frac{-t}{u}\right),\pi -\arctan u \right)} \, \frac{u \, du \,dt  \, ds}{(1+u^2)^{\frac{3}{2}}\sqrt{u^2 -t^2}} 
\end{align*} 
since $g$ is even. Letting $\omega = -u $ in the last integral yields
\begin{align*}
&\langle Rf, g \rangle_{\mathbb{S}^2}\\
&= \frac{1}{\pi}  \int_0^{2\pi} \int_{-\infty}^{\infty} \int_{|t|}^{\infty} f \left( e^{is}, t\right)\overline{g\left(s + \arccos \left( \frac{-t}{u}\right),\arctan u \right)} \, \frac{u \, du \,dt  \, ds}{(1+u^2)^{\frac{3}{2}}\sqrt{u^2 -t^2}}\\ 
&\quad + \frac{1}{\pi}  \int_0^{2\pi} \int_{-\infty}^{\infty} \int^{\infty}_{|t|} f \left( e^{is}, t\right)\overline{g\left(s + \pi+ \arccos \left( \frac{t}{\omega}\right),\pi +\arctan \omega \right)} \, \frac{\omega \, d\omega \,dt  \, ds}{(1+\omega^2)^{\frac{3}{2}}\sqrt{\omega^2 -t^2}}.
\end{align*}

Note that the integration of $g$ is over angles of the form $(\theta, \pi + \rho)$ for some $\rho \in [0,\pi]$. By allowing the second coordinate to be greater than or equal to $\pi$, we lose unique representation of points in $\mathbb{S}^2$, just as when we allow the first coordinate to be outside $[0,2\pi)$. However, this does not affect the integration process. In fact, a point on $\mathbb{S}^2$ described in spherical coordinates by  $\left(s + \pi+ \arccos \left( \frac{t}{\omega}\right),\pi +\arctan \omega \right)$ yields
\begin{align*}
\begin{bmatrix}
\cos(s + \pi+ \arccos \left( \frac{t}{\omega}\right))\sin(\pi +\arctan \omega)\\
\sin(s + \pi+ \arccos \left( \frac{t}{\omega}\right))\sin(\pi +\arctan \omega)\\
\cos(\pi +\arctan \omega)
\end{bmatrix} &= -\begin{bmatrix}
\cos(s + \pi+ \arccos \left( \frac{t}{\omega}\right))\sin(\arctan \omega)\\
\sin(s + \pi+ \arccos \left( \frac{t}{\omega}\right))\sin(\arctan \omega)\\
\cos(\arctan \omega)
\end{bmatrix}\\
&=\begin{bmatrix}
\cos(s + \arccos \left( \frac{t}{\omega}\right))\sin(\pi -\arctan \omega)\\
\sin(s + \arccos \left( \frac{t}{\omega}\right))\sin(\pi -\arctan \omega)\\
\cos(\pi-\arctan \omega)
\end{bmatrix}
\end{align*}
or
\begin{align*}
\left(s + \pi+ \arccos \left( \frac{t}{\omega}\right),\pi +\arctan \omega \right) = \left(s + \arccos \left( \frac{t}{\omega}\right),\pi -\arctan \omega \right).
\end{align*}
Since $g$ is an even function, integrating over points
\begin{align*}
\left(s + \pi+ \arccos \left( \frac{t}{\omega}\right),\pi +\arctan \omega \right)
\end{align*}is equivalent to integrating over points
\begin{align*}
\left(s + \pi+ \arccos \left( \frac{t}{\omega}\right),\arctan \omega \right).
\end{align*} 
Yielding
\begin{align*}
&\langle Rf, g \rangle_{\mathbb{S}^2} \\
&= \frac{1}{\pi}  \int_0^{2\pi} \int_{-\infty}^{\infty} \int_{|t|}^{\infty} f \left( e^{is}, t\right)\overline{g\left(s + \arccos \left( \frac{-t}{u}\right),\arctan u \right)} \, \frac{u \, du \,dt  \, ds}{(1+u^2)^{\frac{3}{2}}\sqrt{u^2 -t^2}} \\
&\quad + \frac{1}{\pi}  \int_0^{2\pi} \int_{-\infty}^{\infty} \int^{\infty}_{|t|} f \left( e^{is}, t\right)\overline{g\left(s + \pi+ \arccos \left( \frac{t}{\omega}\right),\arctan \omega \right)} \, \frac{\omega\, d\omega \,dt  \, ds}{(1+\omega^2)^{\frac{3}{2}}\sqrt{\omega^2 -t^2}} \\
&=\frac{1}{\pi}  \int_0^{2\pi} \int_{-\infty}^{\infty} \int^{\infty}_{|t|} f \left( e^{is}, t\right) \sum_{\sigma = \pm 1} \, \overline{g\left(s + \sigma \arccos \left( \frac{-t}{u}\right),\arctan u \right)} \, \frac{u \, du \,dt  \, ds}{(1+u^2)^{\frac{3}{2}}\sqrt{u^2 -t^2}} .
\end{align*}

Letting $v = \sqrt{u^2 - t^2}$ yields
\begin{align*}
&\langle Rf, g \rangle_{\mathbb{S}^2}\\ &=\frac{1}{\pi}  \int_0^{2\pi} \int_{-\infty}^{\infty} \int^{\infty}_{0} f \left( e^{is}, t\right) \sum_{\sigma = \pm 1}\, \overline{g\left(s + \sigma \arccos \left( \frac{-t}{\sqrt{v^2 + t^2}}\right),\arctan \sqrt{v^2 + t^2} \right)} \, \frac{\, dv \,dt  \, ds}{(1+v^2 +t^2)^{\frac{3}{2}}} .
\end{align*}
Consequently, for all $s\in [0,2\pi)$ and $t\in \mathbb{R}$
\begin{align*}
R^*g(e^{is},t)&=\frac{1}{\pi} \sum_{\sigma = \pm 1} \, \int^{\infty}_{0} g\left(s + \sigma \arccos \left( \frac{-t}{\sqrt{v^2 + t^2}}\right),\arctan \sqrt{v^2 + t^2} \right)\, \frac{dv}{(1+v^2 +t^2)^{\frac{3}{2}}}.
\end{align*}
\end{proof}
As the double fibration theory \cite{GelfDoubleFibration,QuintoependenceGeneralizedRadon} guarantees, $R^*$ integrates the function $g$ over all normal directions $\zeta(\theta,\rho)$ belonging to sets $E_{\zeta(\theta,\rho)}$ that contain the point $(e^{is},t)$ (\Cref{figura2}). To exhibit this relationship, we can show that $(e^{is},t)$ belongs to all planes defined by the normal vectors $\zeta(\theta,\rho)$ we are integrating over. In spherical coordinates, the normal vectors $\zeta(\theta,\rho)$ can be expressed as
\begin{align*}
&\zeta \left( s \pm \arccos \left( \frac{-t}{\sqrt{v^2 +t^2}} \right),\arctan \sqrt{v^2 +t^2}\right)\\ 
&\qquad =\left(\sin\left(\arctan \sqrt{v^2 +t^2}\right) \, e^{i\left(s \pm \arccos\left( \frac{-t}{\sqrt{v^2 +t^2}}\right)\right)}, \cos\left(\arctan \sqrt{v^2 +t^2}\right)\right)
\end{align*}
From which it follows that
\begin{align*}
&(e^{is},t)\cdot \zeta \left( s \pm \arccos \left( \frac{-t}{\sqrt{v^2 +t^2}} \right),\arctan \sqrt{v^2 +t^2}\right)\\
&\qquad =\sin \left( \arctan\sqrt{v^2 +t^2}\right) \mathfrak{Re} \left( e^{is -i\left( s \pm \arccos \left( \frac{-t}{\sqrt{v^2 +t^2}} \right) \right)}\right)\\
&\qquad \quad +t\cos\left( \arctan \sqrt{v^2 +t^2} \right)\\
&\qquad=\frac{\sqrt{v^2 +t^2}}{\sqrt{1+v^2 +t^2}} \mathfrak{Re} \left( e^{\mp i \arccos \left( \frac{-t}{\sqrt{v^2 +t^2}} \right)} \right) +\frac{t}{\sqrt{1+v^2 +t^2}}\\
&\qquad =\frac{\sqrt{v^2 +t^2}}{\sqrt{1+v^2 +t^2}} \cos \left( \mp \arccos \left( \frac{-t}{\sqrt{v^2 +t^2}} \right) \right) +\frac{t}{\sqrt{1+v^2 +t^2}}\\
&\qquad =0.
\end{align*}
Therefore, for all $v \in [0,\infty) $
\begin{align*}
(e^{is},t) \in E_{\zeta \left( s \pm \arccos \left( \frac{-t}{\sqrt{v^2 +t^2}} \right),\arctan \sqrt{v^2 +t^2}\right) }.
\end{align*}

\begin{figure}\label{figura2}
\begin{center}
 \includegraphics[scale=0.325]{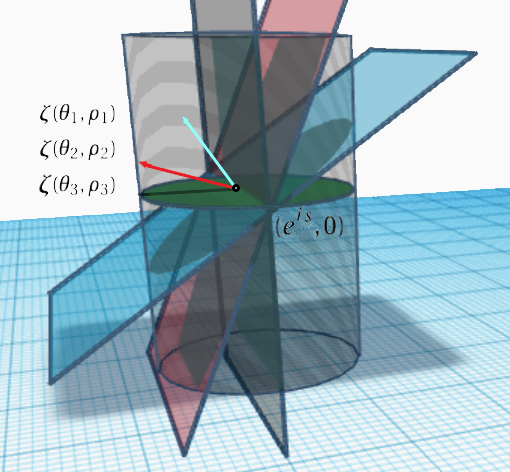}
\end{center}
\caption{$R^*g(e^{is},0)$ integrates the function $g$ over all sets $E_{\zeta(\theta,\rho)}$ containing the point $(e^{is},0)$}
\end{figure}

Assuming $g$ is a well-behaved function, we can recover the original function $g$ from its mean values $R^*g$. We will first compute the forward equation $R^*g = h$ in terms of the Fourier coefficients of $R^*g$ and $g$. Once we have this expression, and noting the relationship between the Fourier coefficients of $R^*g$ and the right-sided Chebyshev fractional integrals, we will state conditions on $g$ for which $R^*g=h$ can be inverted.  

\begin{lemma}\label{lemmaforwardeqgstar}
Let $g \in C(\mathbb{S}^2)$ be an even function and 
\begin{align*}
g(\theta,\rho)=\sum_{n \in \mathbb{Z}} G_n(\rho)e^{in\theta}
\end{align*}
its Fourier series representation for all $(\theta,\rho) \in \Xi$. Let $R^*g(e^{is},t)
= \sum_{n\in \mathbb{Z}} H_n(t)e^{ins}$ be the Fourier series representation of $R^*g$ on $\mathbb{S}\times \mathbb{R}$. Then, for all $n\in \mathbb{Z}$ 
\begin{align}\label{ForwardeqRstar}
H_n(t) &=\frac{(-1)^{|n|}}{\sqrt{\pi}} \left( \Upsilon_{-}^{|n|} G_n^{\#} \right)(t)
\end{align}
where $ \Upsilon_{-}^{|n|}$ is the right-sided Chebyshev fractional integral of order $|n|$ in \cref{right-sided} in \Cref{sec:suptheory} and 
\begin{align*}
G_n^{\#}(v)=\frac{G_{n}\left(\arctan v\right)}{(1+v^2)^{\frac{3}{2}}}.
\end{align*}
\end{lemma}

\begin{proof}
If $n\in \mathbb{Z}$, then using the formula for $R^*$ from \Cref{Dual}
\begin{align*}
&H_n(t) = \frac{1}{2\pi} \int_0^{2\pi} R^*g(e^{is},t)e^{-ins}\,ds\\
&\qquad =\frac{1}{2\pi^2}  \sum_{\sigma = \pm 1}  \int_0^{2\pi} \, \int^{\infty}_{0} g\left(s + \sigma \arccos \left( \frac{-t}{\sqrt{v^2 + t^2}}\right),\arctan \sqrt{v^2 + t^2} \right)\, \frac{e^{-ins}dv \, ds}{(1+v^2 +t^2)^{\frac{3}{2}}} \\
&\qquad =\frac{1}{2\pi^2}  \sum_{\sigma = \pm 1}  \int^{\infty}_{0} \, \int_0^{2\pi} g\left(s + \sigma \arccos \left( \frac{-t}{\sqrt{v^2 + t^2}}\right),\arctan \sqrt{v^2 + t^2} \right)\, \frac{e^{-ins} ds \, dv}{(1+v^2 +t^2)^{\frac{3}{2}}} .
\end{align*}
Letting $u = s +\sigma \arccos \left( \frac{-t}{\sqrt{v^2 + t^2}} \right)$ with respect to the $s$ variable and noting that the integrand is a $2\pi-$periodic function of $u$ yields
\begin{align*}
H_n(t) &=\frac{1}{2\pi^2}  \sum_{\sigma = \pm 1}  \int^{\infty}_{0} \, \int_{0}^{2\pi} g\left(u,\arctan \sqrt{v^2 + t^2} \right)\, \frac{e^{-in\left( u -\sigma \arccos \left( \frac{-t}{\sqrt{v^2 + t^2}} \right) \right)} du \, dv}{(1+v^2 +t^2)^{\frac{3}{2}}}. 
\end{align*}
This way,
\begin{align*}
H_n(t)&=\frac{1}{\pi}  \sum_{\sigma = \pm 1}  \int^{\infty}_{0} \left( \frac{1}{2\pi} \int_{0}^{2\pi} g\left(u,\arctan \sqrt{v^2 + t^2} \right) e^{-inu} \, du \right) \, \frac{ e^{in\sigma \arccos \left( \frac{-t}{\sqrt{v^2 + t^2}} \right)}\, dv}{(1+v^2 +t^2)^{\frac{3}{2}}}\\
&=\frac{1}{\pi}  \sum_{\sigma = \pm 1}  \int^{\infty}_{0} G_n\left(\arctan \sqrt{v^2 +t^2}\right)\, \frac{e^{in\sigma \arccos \left( \frac{-t}{\sqrt{v^2 + t^2}} \right)}\, dv}{(1+v^2 +t^2)^{\frac{3}{2}}} .
\end{align*}

Using Euler's formula and noticing that
\begin{align*}
\sin \left( n \sigma  \arccos \left( \frac{-t}{\sqrt{v^2 + t^2}} \right) \right) &= \sigma \sin \left( n\arccos \left( \frac{-t}{\sqrt{v^2 + t^2}} \right) \right)\\
\cos \left( n \sigma  \arccos \left( \frac{-t}{\sqrt{v^2 + t^2}} \right) \right) &= \cos \left( n\arccos \left( \frac{-t}{\sqrt{v^2 + t^2}} \right) \right)
\end{align*}
yields
\begin{align*}
H_n(t)&=(-1)^{|n|}\frac{2}{\pi} \int^{\infty}_{0} G_n\left(\arctan \sqrt{v^2 +t^2}\right) \, \frac{T_{|n|}\left( \frac{t}{\sqrt{v^2 + t^2}} \right)}{(1+v^2 +t^2)^{\frac{3}{2}}} \, dv
\end{align*}
 
Letting $u=\arctan{\sqrt{v^2+t^2}}$ we obtain 
\begin{align*}
H_n(t)
&=(-1)^{|n|}\frac{2}{\pi} \int_{\arctan |t|}^{\frac{\pi}{2}}  G_{n}(u)\frac{ T_{|n|}\left( \frac{t}{\tan u } \right) \tan u}{\sqrt{1+\tan^2 u} \sqrt{\tan^2 u - t^2}} \,du.
\end{align*}
Making the change of variables $\omega = \tan u$ yields
\begin{align*}
H_n(t) &=(-1)^{|n|}\frac{2}{\pi} \int_{|t|}^{\infty}  G_{n}(\arctan \omega )\frac{\omega}{(1+\omega^2)^{\frac{3}{2}} \sqrt{\omega^2 - t^2}} T_{|n|} \left( \frac{t}{\omega} \right) \,d\omega.
\end{align*}
Letting
\begin{align}\label{Ghashtag}
G_n^{\#}(\omega)=\frac{ G_{n}\left(\arctan \omega \right)}{(1+\omega^2)^{\frac{3}{2}}},
\end{align}
we can see that
\begin{align*}
H_n(t) &=\frac{(-1)^{|n|}}{\sqrt{\pi}} \left( \Upsilon_{-}^{|n|} G_n^{\#} \right)(t),
\end{align*}
where $\Upsilon_-^{|n|}$ is the right-sided Chebyshev fractional integral as defined in \cref{right-sided} in \Cref{sec:suptheory}.
\end{proof}

\subsection*{Null Space, Injectivity of $\mathbf{R}^*$ and an Inversion Formula}
We will show that if a nontrivial, even, and continuous function $g$ belongs to the null space of $R^*$, then $g$ would be unbounded on the equator or at the poles  of $\mathbb{S}^2$. This contradiction implies that $R^*$ is injective in the space of continuous and even functions on $\mathbb{S}^2$.

Now, 
\begin{align*}
R^*g(e^{is},t)=0 &\iff \forall n \in \mathbb{Z} \quad H_n(t) =0\\
&\iff  \forall n \in \mathbb{Z} \quad \frac{(-1)^{|n|}}{\sqrt{\pi}}\left( \Upsilon_{-}^{|n|} G_n^{\#} \right)(t) = 0.
\end{align*}
Assuming $G^\#_n$ fulfills the hypotheses in \Cref{upsmenosinj} and \Cref{upsmenos} and $R^*g = 0$, then, by the aforementioned theorems, if $|n|<2$ then $G_n^{\#}  =0$, implying that
\begin{align*}
G_{n}(\rho) =  0
\end{align*}
for all $\rho\in (0,\pi/2)$. If $|n| \geq 2$, then
\begin{align*}
G_n^{\#}(\omega) =  \sum_{k=0}^{\lfloor \frac{|n|}{2} \rfloor-1} c^n_k \omega^{2k-|n|}
\end{align*}
for $\omega \in (0,\infty) $ and some coefficients $c^n_k \in \mathbb{C}$. Therefore, 
\begin{align*}
G_n^{\#}(\omega) =  \sum_{k=0}^{\lfloor \frac{|n|}{2} \rfloor-1} c^n_k \omega^{2k-|n|} &\iff
G_{n}(\arctan \omega) =  \sum_{k=0}^{\lfloor \frac{|n|}{2} \rfloor-1} c^n_k \omega^{2k-|n|}\left(1+\omega^2\right)^{\frac{3}{2}}\\
 &\iff G_{n}(\rho) =  \sum_{k=0}^{\lfloor \frac{|n|}{2} \rfloor-1} c^n_k \frac{\left(1+\tan^2\rho\right)^\frac{3}{2}}{\tan^{|n|-2k}\rho}  
\end{align*}
for $\rho \in (0,\frac{\pi}{2}) $.

\begin{theorem}\label{injectivityR*}
The dual transform $R^*$ is injective in $C_e(\mathbb{S}^2)$.
\end{theorem}

\begin{proof}
Let $g \in C(\mathbb{S}^2)$ be an even function and, without loss of generality, assume $ 0 < \rho <\pi/2 $. As $g$ is a continuous function, its Fourier coefficients $G_n(\rho)$ are also continuous on $[0,\pi/2]$. We begin by showing that $G_n^\# \in L_{\chi^-_{|n|}}((0,\infty))$, that is $G_n^\#$ satisfies \cref{lchiminus} in \Cref{sec:suptheory} for all $n\in \mathbb{Z}$. Let $n \in \mathbb{Z}$ and $a_1>0 $. If $n$ is even then
\begin{align*}
\int_{a_1}^\infty \left| G_n^\# (v)\right|\,dv &= \int_{a_1}^\infty \frac{G_{n}(\arctan v)}{(1+v^2)^{\frac{3}{2}}}\, dv\\
&= \int_{\arctan a_1}^{\frac{\pi}{2}} \frac{|G_{n}(u)|}{\sqrt{1+\tan^2u}}\, du
\end{align*}
by letting $u=\arctan v$. Therefore,
\begin{align*}
\int_{a_1}^\infty \left| G_n^\# (v)\right|\,dv &\leq \int_{\arctan a_1}^{\frac{\pi}{2}} \left| G_n (u)\right|\,du\\
&\leq \| G_n\|_1\\
&<\infty.
\end{align*}
If $n$ is odd then
\begin{align*}
\int_{a_1}^\infty \frac{\left| G_n^\#(v) \right| (1+|\log v |)}{v}\, dv &= \int_{a_1}^\infty \frac{|G_n(\arctan v)|(1+|\log v|)}{v (1+v^2)^{\frac{3}{2}}} dv\\
&= \int_{\arctan a_1}^{\frac{\pi}{2}} \frac{|G_n(u)|(1+|\log \tan u| )}{\tan u \, \sqrt{1+\tan^2 u }} du
\end{align*}
by letting $u=\arctan v$. This way
\begin{align*}
\int_{a_1}^\infty \frac{\left| G_n^\#(v) \right| (1+|\log v |)}{v}\, dv 
&\leq \int_{\arctan a_1}^{\frac{\pi}{2}} |G_n(u)|\frac{(1+|\log \tan u| )}{\tan u} du.
\end{align*}
The quantity $\frac{(1+|\log \tan u| )}{\tan u}$ attains the following maximum in $[\arctan a_1, \pi/2)$: 
\begin{align*}
\max_{[\arctan a_1, \pi/2)} \left\lbrace  \frac{(1+|\log \tan u| )}{\tan u} \right\rbrace = \begin{cases} \frac{1}{a_1} + \frac{|\log a_1|}{a_1} &\quad \text{ if } 0<a_1 <1, \\ \frac{1}{a_1} + \frac{1}{e} &\quad \text{ if }1 \leq a_1.\end{cases}
\end{align*}
Therefore,
\begin{align*}
\int_{a_1}^\infty \frac{\left| G_n^\#(v) \right| (1+|\log v |)}{v}\, dv 
&\leq \left(\frac{1}{a_1} + \max \left\lbrace \frac{1}{e},\frac{|\log a_1|}{a_1} \right\rbrace \right)\int_{\arctan a_1}^{\frac{\pi}{2}} |G_n(u)|du\\
&\leq \left(\frac{1}{a_1} + \max \left\lbrace \frac{1}{e},\frac{|\log a_1|}{a_1} \right\rbrace \right) \|G_n\|_1\\
&<\infty.
\end{align*}
Thus proving that $G_n^\# \in L_{\chi^-_{|n|}}^1(0,\infty)$.

However, even and continuous functions in the null space of $R^*$ have Fourier series coefficients
\begin{align*}
G_{n}(\rho) \eqae  \sum_{k=0}^{\lfloor \frac{|n|}{2} \rfloor-1} c^n_k \frac{\left(1+\tan^2\rho\right)^\frac{3}{2}}{ \tan^{|n|-2k}\rho} 
\end{align*}
for all $|n|\geq2$. Consequently, if any of the coefficients $c^n_k \neq 0$, then $g$ would not be continuous on the equator or at the poles of $\mathbb{S}^2$. It follows, that $R^*$ must have a trivial null space when defined in the space of continuous and even functions on $\mathbb{S}^2$.
\end{proof}

\begin{corollary}
The dual transform $R^*$ is injective in $R(H^2_{pc}(\mathbb{S}\times\mathbb{R}))$.
\end{corollary}

\begin{proof}
Let $f \in H_{pc}^2(\mathbb{S}\times \mathbb{R})$. Since $R$ annihilates odd functions we may assume, without loss of generality, that $f$ is even. By \Cref{SobolevM} there exists a bounded and continuous function that is equal to $f$ almost everywhere. We may assume, without loss of generality, that $f$ denotes this continuous representative. Since $f$ is a continuous even function, $Rf$ is continuous and even on $\mathbb{S}^2$. Therefore, by \Cref{injectivityR*}, $R^*$ is injective in $R(H^2_{pc}(\mathbb{S}\times \mathbb{R}))$.

\end{proof}

We conclude by providing an inversion formula.

\begin{theorem}
Let $\Xi$ be as defined in \Cref{Sec:Transform}, $g \in C(\overline{\Xi})$ be an even function on $\mathbb{S}^2$ and 
\begin{align*}
g(\theta,\rho)=\sum_{n \in \mathbb{Z}} G_n(\rho)e^{in\theta}
\end{align*}
its Fourier series representation for all  $(\theta,\rho) \in \Xi$. If for $|n|\geq 2$
\begin{align*}
\int_{\arctan a}^{\frac{\pi}{2}} |G_n(u)| \tan^{|n|-1}u\, du <\infty \qquad \text{ for all a > 0},
\end{align*}
then for almost all  $t \in (0,\infty)$
\begin{align*}
G_{n}(\arctan t)= \begin{cases}
\sqrt{\pi}(1+t^2)^{\frac{3}{2}} \left(\textsc{D}^{\frac{1}{2}}_{-,2} H_0\right)(t) &\qquad \text{ for } n =0,\\
&\qquad \\
-\sqrt{\pi}t(1+t^2)^{\frac{3}{2}} \left( D_{-,2}^{\frac{1}{2}} \frac{H_1}{t} \right)(t)&\qquad \text{ for }|n|=1,\\
&\qquad \\
(-1)^{|n|+1} \frac{\sqrt{\pi}}{2}(1+t^2)^{\frac{3}{2}}\frac{d}{dt}  \left( \overset{*}{\Upsilon}^{|n|}_- t^{-2}H_n\right)(t) &\qquad \text{ for }|n| \geq 2,
\end{cases}
\end{align*}
where
\begin{align*}
H_n(t) &=\frac{(-1)^{|n|}}{\sqrt{\pi}} \left( \Upsilon_{-}^{|n|} G_n^{\#} \right)(t),
\end{align*}
and
\begin{align*}
\textsc{D}^{\frac{1}{2}}_{-,2} \varphi &= -\frac{1}{2} \frac{d}{dt} t I^{\frac{1}{2}}_{-,2} t^{-2} \varphi,\\
\left( I^{\frac{1}{2}}_{-,2} \varphi \right)(t)&=\frac{2}{\sqrt{\pi}} \int_t^\infty \frac{\varphi(r)r}{\sqrt{r^2 -t^2}}dr,\\
\left( \overset{*}{\Upsilon}^m_- \varphi \right)(t)&=\frac{2t}{\sqrt{\pi}}\int_t^\infty\frac{\varphi(r)T_m\left(\frac{r}{t} \right)}{r \, \sqrt{r^2-t^2}} dr
\end{align*} 
as defined in \Cref{sec:suptheory}.
\end{theorem}

\begin{proof}

We prove this using \Cref{Rub244} and \Cref{Rub252} from \Cref{sec:suptheory}. First notice that if $|n| < 2$ the conditions of the aforementioned theorems imply that $g$ must be such that
\begin{align*}
\int_{\arctan a}^{\frac{\pi}{2}} \frac{|G_n(u)|\tan^{-\nu}u}{\sqrt{1+\tan^2u}}du <\infty \quad \text{ for all } a > 0 \text{ and }\nu=\begin{cases} 0 &\quad \text{ if }n \text{ is even}\\
1 &\quad \text{ if }n \text{ is odd.}\end{cases}
\end{align*}
However, since $g$ is a continuous function on $\mathbb{S}^2$
\begin{align*}
\int_{\arctan a}^{\frac{\pi}{2}} \frac{|G_n(u)|\tan^{-\nu}u}{\sqrt{1+\tan^2u}}du & \int_{\arctan a}^{\frac{\pi}{2}}|G_n(u)| \tan^{-\nu}u\,du\\
&\leq  \max \left\lbrace \frac{1}{a},1 \right\rbrace \, \|G_n(u)\|_1\\
&<\infty.
\end{align*}
Therefore, since for $|n| <2 $ we have that
\begin{align*}
\int_{\arctan a}^{\frac{\pi}{2}} \frac{|G_n(u)|(\tan u)^{-\nu}}{\sqrt{1+\tan^2u}}du <\infty \quad\text{ where }\nu=\begin{cases} 0 &\quad \text{ if }n \text{ is even}\\
1 &\quad \text{ if }n \text{ is odd,}\end{cases}
\end{align*}
and 
\begin{align*}
\int_{\arctan a}^{\frac{\pi}{2}} \frac{|G_n(u)| \tan^{|n|-1}u}{\sqrt{1+\tan^2 u}} du &\leq \int_{\arctan a}^{\frac{\pi}{2}}|G_n(u)| \tan^{|n|-1}u\, du \\
&<\infty
\end{align*}
for $|n|\geq 2$ and all $a>0$, by \Cref{Rub244} and \Cref{Rub252} $G_n(t)$ can be uniquely reconstructed almost everywhere from the mean values
\begin{align}\label{forwardH}
H_n(t) &=\frac{(-1)^{|n|}}{\sqrt{\pi}} \left( \Upsilon_{-}^{|n|} G_n^{\#} \right)(t).
\end{align}
Moreover,
\begin{align*}
G_{n}(\arctan t)= \begin{cases}
\sqrt{\pi}(1+t^2)^{\frac{3}{2}} \left(\textsc{D}^{\frac{1}{2}}_{-,2} H_0\right)(t) &\qquad \text{ for } n =0,\\
&\qquad \\
-\sqrt{\pi}t(1+t^2)^{\frac{3}{2}} \left( D_{-,2}^{\frac{1}{2}} \frac{H_1}{t} \right)(t)&\qquad \text{ for }|n|=1,\\
&\qquad \\
(-1)^{|n|+1} \frac{\sqrt{\pi}}{2}(1+t^2)^{\frac{3}{2}}\frac{d}{dt}  \left( \overset{*}{\Upsilon}^{|n|}_- t^{-2}H_n\right)(t) &\qquad \text{ for }|n| \geq 2,
\end{cases}
\end{align*}
\end{proof}

We would like to point out that, analogously to \textit{Remark 2.53} from \cite{RubFractional}, we may relax the assumptions on $G_n$ for $|n| \geq 2$ by assuming only that $g$ is a continuous function on $\mathbb{S}^2$. In this case, we can find non-unique solutions to the forward equation
\begin{align*}
H_n(t) &=\frac{(-1)^{|n|}}{\sqrt{\pi}} \left( \Upsilon_{-}^{|n|} G_n^{\#} \right)(t)
\end{align*}
for all $|n| \geq 2$.

\section*{Acknowledgments}
The author is thankful to Fulton Gonzalez and Eric Todd Quinto for all their suggestions, guidance, and time. Their comments and help were invaluable. 

\bibliographystyle{siamplain}
\bibliography{MasterBibliography_arxiv}
\end{document}